\documentclass[12pt,epsfig,tikz,multi]{amsart}
\usepackage{url}
\usepackage{verbatim}
\usepackage{latexsym}
\usepackage{hyperref}
\usepackage{amsrefs}
\usepackage{amsthm}
\usepackage{breqn}
\usepackage{float}
\restylefloat{table}
\usepackage{amsmath,amscd,amssymb}
\usepackage{graphicx}
\usepackage{caption}
\usepackage{subcaption}
\usepackage{multirow}
\usepackage{enumitem}
\usepackage{xcolor}
\usepackage{graphicx,wrapfig,lipsum}
\usepackage{tikz}
\usepackage{diagbox}
\usepackage{graphicx}
\usepackage{float}
\usepackage{import}
\usepackage{xifthen}
\usepackage{pdfpages}
\usepackage{transparent}
\usepackage{amsmath, amssymb, mathtools}
\usepackage{graphicx}
\usepackage{float}
\usepackage{import}
\usepackage{xifthen}
\usepackage{pdfpages}
\usepackage{transparent}

 2

\newtheorem{theorem}{Theorem}[section]
\newtheorem{prop}{Proposition}[section]

\newtheorem{definition}[theorem]{Definition}

\newtheorem{corollary}[theorem]{Corollary}

\newtheorem{fact}[theorem]{Fact}

\numberwithin{equation}{section}
\def \C {{\mathbb {C}}}
\def \R {{\mathbb {R}}}
\def \H {{\mathbb {H}}}

\title{Higher genus maxfaces with Enneper end}

\author[R. Bardhan]{Rivu Bardhan}
\address{Department of Mathematics, Shiv Nadar Institute of Eminence, Deemed to be University, Dadri 201314, Uttar Pradesh, India.}
\email{rb212@@snu.edu.in}

\author[I. Biswas]{Indranil Biswas}
\address{Department of Mathematics, Shiv Nadar Institute of Eminence, Deemed to be University, Dadri 201314, Uttar Pradesh, India.}
\email{indranil.biswas@snu.edu.in, inrdranil29@gmail.com}

\author[P. Kumar]{Pradip Kumar}
\address{Department of Mathematics, Shiv Nadar Institute of Eminence, Deemed to be University, Dadri 201314, Uttar Pradesh, India.}
\email{pradip.kumar@snu.edu.in}


\begin{document}
\include{plaquettemacro}
\maketitle
\begin{abstract}
We have proven the existence of new higher-genus maxfaces with Enneper end. These maxfaces are not the 
companions of any existing minimal surfaces, and furthermore, the singularity set is located away from the 
ends. The nature of the singularities is systematically investigated.
\end{abstract}

\section{\textbf{Introduction}}

Analogous to minimal surfaces in $\mathbb{R}^3$, maximal surfaces are immersions with zero mean curvature 
in the Lorentz Minkowski space $\mathbb E_1^3$. These surfaces arise as solutions to the variational 
problem of locally maximizing the area among space-like surfaces. There are similarities between maximal 
surfaces and minimal surfaces, apart from having zero mean curvature, for example, both admit a
Weierstrass Enneper representation. However, striking differences emerge in the global study of these 
surfaces.  While the only complete maximal surfaces are planes, there are many complete minimal surfaces 
apart from planes, such as the Catenoid, Enneper surface, Costa surfaces, etc.

The global existence of maximal surfaces warrants allowing natural singularities. We study
a particular class of maximal surfaces called maxfaces, which were named so by Umehara and Yamada in \cite{UMEHARA2006}. As there are
isolated singularities as well, and that are not allowed on maxfaces.  Maxfaces
 have only non-isolated singularities, and at singularities, limiting tangent vectors 
contain a light-like vector.
The Lorentzian Catenoid is a maxface of genus zero. Imaizumi and Kato \cite{imaizumi2008} classified genus-zero maxfaces. Kumar and 
Mohanty \cite{Saipradip2023} have shown the existence of genus-zero maxfaces with a prescribed singular set and an arbitrary number of complete and simple ends.

For higher genus, Kim and Yang \cite{Kim2006} have proved the existence of genus one maxfaces and 
have also shown the existence of maximal maps (not the maxfaces) for higher genus. These maximal maps and maxfaces have 
two ends --- both of which are catenoid types. Furthermore, Fujimori, Rossman, Umehara, Yang, and Yamada, 
\cite{Fujimori2009}, have constructed a family of complete maxfaces, denoted as $f_k$, $k \,=\, 1,\, 2,\, 
3,\,\cdots$, with two ends. The maxface $f_k$ is of genus $k$ if $k$ is odd and it has genus $\frac{k}{2}$ 
if $k$ is even. All of them have two ends.

In 2016, authors of \cite{Fujimori2016highergenus} constructed maxfaces of any odd genus $g$ with two complete ends (if $g\,=\,1$, the 
ends are embedded) and maxfaces of genus $g\,=\,1$ with three complete embedded ends. 

Here, we focus on higher genus ($g \,\geq\, 2$) maxfaces with Enneper end and prove the existence of higher 
genus maxfaces with one Enneper end. We also analyze the nature of singularities on these surfaces.

Due to the presence of the singular set, we cannot directly apply methods from minimal surfaces to prove the existence of a maxface. Although the Weierstrass-Enneper representation differs only by a sign, the question regarding the existence of maxfaces of arbitrary genus, with a prescribed singular set and a specified nature of singularity, remains inadequately explored. The main challenges can be distilled into two categories: the construction of the maxface and the comprehension of the singularities' nature.

\subsection{Challenges in Constructing Maxfaces in  comparison to Minimal Surfaces in $\mathbb{R}^3$}
For constructing maxfaces of higher genus, the obvious initial approach is to consider companion surfaces. Suppose for a minimal 
surface the Weierstrass data on a Riemann surface $M$ is $\lbrace g,\,dh\rbrace$. Then if the companion exists, the corresponding maxface 
would have the data $\lbrace- i g,\, i dh\rbrace$. However, this approach faces two main challenges:

\textbf{1.} It is possible that the singularity set may extend towards the ends, preventing the creation of a complete maxface. For instance, consider the Weierstrass data for Jorge-Meek's minimal surface, given by:
\[
g(z) \,=\, z^n \ \ \ \text{ and } \ \ \ \omega \,=\, \frac{dz}{(z^{n+1}-1)^2}.
\]
This data results in a complete minimal surface on $\mathbb{C}\cup\{\infty\}$ with punctures at $\{1,\, \zeta, 
\,\zeta^2,\, \cdots ,\, \zeta^{n-1}\}$, where $\zeta\,=\, \exp\left({\frac{2 i \pi}{n}}\right)$. It is 
observed that the companion, whose Gauss map is $g_0\,=\, -i g$, gives a maximal map, but it is not a 
complete maxface because the singular set $\{z\,\, \big\vert\,\,|g_0(z)| \,=\,1\}$ is not compact.  We refer 
to Fact \ref{fact} for the definition of a complete maxface.

\textbf{2.} The second issue concerns the solvability of the period problem in a direct way. For example, consider the Costa surface:
\[
M \,=\, \left\{(z,\,w) \in \mathbb{C}\times\mathbb{C}\cup\{(\infty,\infty)\} \, \big\vert\, \, w^2\,=\,z(z^2-1) \right\}
\setminus \{(0,\,0),\,(\pm 1,\,0)\}
\]
with data $\left\{\frac{a}{w},\,\frac{2a}{z^2-1}dz\right\}$, $a\,\in\,\mathbb{R}^{+}\setminus\{0\}$.

Suppose there exists a maxface for which the companion is the Costa surface. Then its data should be $\{M,\, -\frac{i a}{w},\, \frac{2i 
a}{z^2-1}dz \}$. Let $\tau$ be an one-sheeted loop around $(-1,\,0)$ that does not contain $(1,\,0)$. Then $\int_\tau(\frac{2i 
a}{z^2-1}dz)\,=\,2i a$, $a\,\neq\, 0$. Therefore, for the corresponding maxface, the period problem is not solved. Hence, there does not exist 
any maxface for which the corresponding companion is the Costa surface.

\begin{figure}[h]
    \centering
    \def\svgwidth{0.7\columnwidth}
\begingroup%
  \makeatletter%
  \providecommand\color[2][]{%
    \errmessage{(Inkscape) Color is used for the text in Inkscape, but the package 'color.sty' is not loaded}%
    \renewcommand\color[2][]{}%
  }%
  \providecommand\transparent[1]{%
    \errmessage{(Inkscape) Transparency is used (non-zero) for the text in Inkscape, but the package 'transparent.sty' is not loaded}%
    \renewcommand\transparent[1]{}%
  }%
  \providecommand\rotatebox[2]{#2}%
  \newcommand*\fsize{\dimexpr\f@size pt\relax}%
  \newcommand*\lineheight[1]{\fontsize{\fsize}{#1\fsize}\selectfont}%
  \ifx\svgwidth\undefined%
    \setlength{\unitlength}{418.49999135bp}%
    \ifx\svgscale\undefined%
      \relax%
    \else%
      \setlength{\unitlength}{\unitlength * \real{\svgscale}}%
    \fi%
  \else%
    \setlength{\unitlength}{\svgwidth}%
  \fi%
  \global\let\svgwidth\undefined%
  \global\let\svgscale\undefined%
  \makeatother%
  \begin{picture}(1,0.80107531)%
    \lineheight{1}%
    \setlength\tabcolsep{0pt}%
    \put(0,0){\includegraphics[width=\unitlength,page=1]{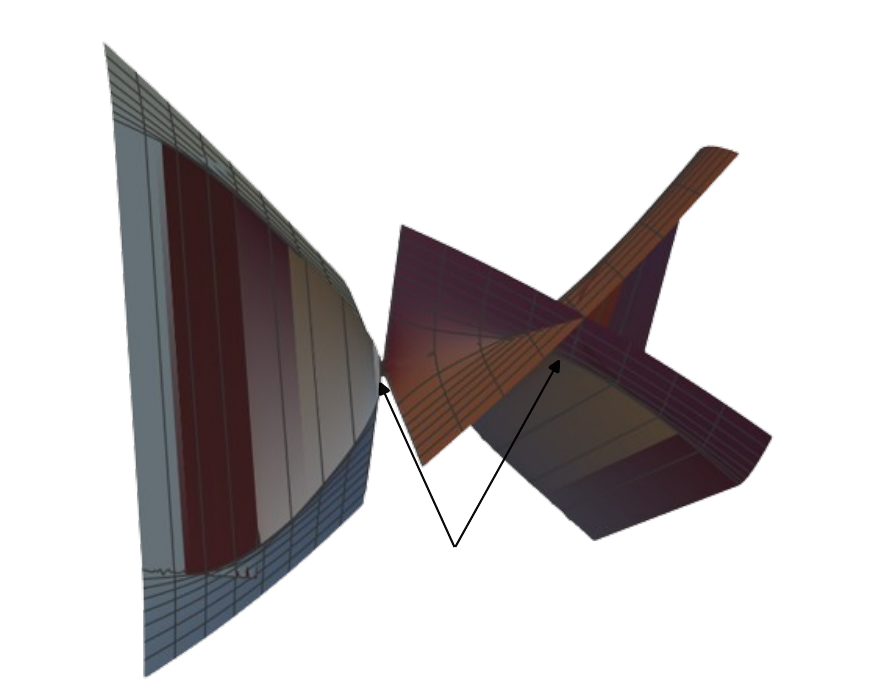}}%
    \put(0.40057419,0.14859454){\color[rgb]{0,0,0}\transparent{0.95815808}\makebox(0,0)[lt]{\lineheight{0.1}\smash{\begin{tabular}[t]{l}\text{These are part of two singular curves.} \\\\\\\\\\\\\\\\\\\\\text{It has at least four swallowtails.}\end{tabular}}}}%
  \end{picture}%
\endgroup%

    \caption{Half of a Lorentzian Chen–Gackstatter surface: A visualization using Mathematica with the Weierstrass data derived from a genus-one zigzag or tweezer. }
    \label{fig:Itroduction}
\end{figure}
To construct a complete maxface with a prescribed singular set, it is common practice to forego the companion approach and address the period problem from the ground up. This involves taking into account the singularities to ensure completeness. In this article, we employ the method of solving the period problem as proposed by Wolf and Weber \cite{weber1998minimal,Weber1998TeichmullerTA}.  Using this method, we prove the existence of higher genus maxfaces with an Enneper end. Notably, these new maxfaces are not companions to existing higher genus minimal surfaces with an Enneper end, as described in \cite{weber1998minimal}. Precisely, we have proven:
\subsection{(Theorem \ref{prop:maxfacefromZigzag})}\label{firstresult} Given a reflexive zigzag of genus $p$, there exists a maxface $\widetilde{X}$ of genus $p$ with one Enneper end defined on a Riemann surface that is `generated' by the zigzag.   This result is along a similar line as in \cite{weber1998minimal}, where authors proved the existence of a reflexive zigzag of any genus and then the existence of a minimal surface $X$ for a reflexive zigzag.
 \subsection{(Theorem \ref{thm:surfaceFromTweezer})}\label{secondresult} Given a reflexive tweezer $T$ of genus $p$, there exists a minimal surface $X_T$ and a maxface $\widetilde{X}_T$ of genus $p$, each having one Enneper end and at most eight symmetries. Furthermore, $X_T$ (respectively, $\widetilde{X}_T$) is not symmetric to the minimal surface $X$ (respectively, $\widetilde{X}$) as discussed above and  in Section \ref{subsection:minimal surface weber & wolf}. Moreover, a reflexive tweezer exists for any genus.

In a complete maxface, which is not a plane, numerous types of singularities appear. For example, in the  Lorentzian Enneper surface, there are four swallowtails, and the rest are cuspidal edges (see 
\cite{UMEHARA2006}). In the context of our surfaces mentioned above, it is natural to ask about the singular set and the types of singularities that might emerge. 

Given that the Gauss map is defined via the Schwartz-Christoffel map and the existence of such a surface is based on the fact that there is a reflexive tweezer or zigzag, the surface data is not explicitly known to us. Consequently, pinpointing the exact nature of the singularity becomes a complex endeavor. In Section 7, we delve into a comprehensive analysis of the possible singularities that might emerge on these surfaces. Specifically, we prove the following:
\subsection{(Theorem \ref{thm:Sing})} Let $X$ be a maxface of genus $p$ defined via a zigzag. Moreover, if this maxface is a front, it possesses $p+1$ connected components of the singular set, each of which is topologically a circle. Each component will have:
    \begin{enumerate}
        \item At least two points where we have either swallowtails, cuspidal butterflies, or the special singular points of type 1.

    \item At least two points where we have either cuspidal cross-caps, cuspidal $S_1^-$, or special singular points of type 2.

    \item Other singularities that are not any of the above are cuspidal edges.
    \end{enumerate}
Furthermore, if the tweezer is of genus $1$, then the maxface from the tweezer and zigzag are the same. For genus $p\geq 2$, if the tweezer has data \((c, t_1, \ldots, t_p)\) with \(t_{-j} = -t_j\), then it possesses at least $p-1$ connected components of the singular set. Moreover, the nature of the singularity in each connected component of singularities is the same as in the case of the zigzag.

Section 2 is preliminary. In Section 3, we recall the construction of minimal surfaces by zigzag, as 
presented in \cite{weber1998minimal}. Using similar methods, in Theorem \ref{prop:maxfacefromZigzag}, 
we construct maxfaces with zigzag. We demonstrate that this maxface is not the companion of the minimal 
surface by zigzag; instead, it is symmetric to a companion of the same minimal surface.

One of our goals is to construct maxfaces distinct from the companions of minimal surfaces by zigzag that are not symmetric to their companions. To achieve this, in Section 4, we introduce a geometric shape known as a `tweezer' (a special ortho-disk in the terminology of \cite{Weber1998TeichmullerTA}). In Theorem \ref{thm:surfaceFromTweezer}, we prove that for any genus $p \geq 2$, if there exists a reflexive tweezer, then there exists a maxface that is neither the companion of the minimal surface derived from the zigzag nor the companion of any symmetric minimal surface derived from the zigzag. Section 6 delves into a discussion regarding the existence of reflexive tweezers. In Section 7, we extensively study the singularities.

\section{\textbf{Preliminaries}}

We recall the Weierstrass-Enneper representation for the minimal surfaces in $\mathbb R^3$ and maximal 
surfaces as well as maxfaces in the Lorentz Minkowski space $\mathbb E_1^3$. Here $\mathbb{E}_1^3$ is the vector space 
$\mathbb{R}^3$ with the bilinear form $dx^2+dy^2-dz^2$.

\subsection{Weierstrass-Enneper representation for minimal surfaces in $\mathbb R^3$}

For an oriented minimal surface $X\,:\, M\,\longrightarrow\,
{\mathbb R}^3$, there is a natural Riemann surface structure on $M$ together with a
meromorphic function $g$ as well as a holomorphic one-form $dh$ on $M$ such that the poles and zeros of $g$ match 
with the zeros of $dh$ with the same order, and
\begin{equation}
X(p) \,=\, \text{Re}\, X(x_0)+ \text{Re} \int^p_{x_0} \left(\frac{1}{2} (g^{-1} - g)  , \frac{ i }{2} ( g^{-1} + g ), 1 \right)dh.
\end{equation}

The triple $\lbrace M,\,g,\,dh \rbrace$ is referred to as the Weierstrass data for the minimal surface $(X,\,
M)$. Furthermore, with such $g$ and $dh$ on a Riemann surface $M$, if the above integral is well defined, then it
is a minimal immersion in ${\mathbb R}^3$.

Now, let us move to the maximal immersions.

\subsection{Weierstrass-Enneper representation for the maximal map}\label{weirstrassdata}
Let $M$ be a Riemann surface, $g$ a meromorphic function on $M$ and $dh$ a holomorphic 1-form on $M$, such
that the following conditions are satisfied:
\begin{enumerate}
\item If $p\,\in\, M$ is a zero or pole of $g$ of order $m$, then $dh$ has a zero at $p$ of order at least $m$,

\item $|g|$ is not identically equal to 1 on $M$,

\item for all closed loops $\gamma$ on $M$,
    \begin{equation}\label{PeriodProblem}
    \int_{\gamma}gdh+\overline{\int_{\gamma}g^{-1}dh}\,=\,0,\ \ \  Re\int_{\gamma}dh\,=\,0.
    \end{equation} 
\end{enumerate}
Then the map $X\,:\, M\, \longrightarrow\, \mathbb E_1^3$ defined by
\begin{equation}\label{maximal_map}
X(p)\,=\,\text{Re}\int_{x_0}^p\left( \frac{1}{2}(g^{-1} + g), \frac{ i }{2}(g^{-1} - g), 1 \right) dh
\end{equation}
is a maximal map with the base point $x_0\,\in\, M$ \cite{Estudillo1992,UMEHARA2006}. Furthermore, any
maximal map can be expressed in this form. The triple $(M,\,g,\, dh)$ constitutes the Weierstrass data
for the maximal map. 

The pullback metric on the Riemann surface $M$ is given by 
$$ds^2\,=\,\frac{1}{4}\left(|g|^{-1}-|g|\right)^2|dh|^2,$$ and the singular locus of the maximal map with
Weierstrass data $(M,\,g,\, dh)$ is the subset $\{p\,\in\, M\,\big\vert\,\, |g(p)|\,=\,1\, \text{ or }\,
dh(p)\,=\,0\}$.

In the context of \textbf{maxfaces}, it is known that
\begin{equation}\label{divisor}(g)_0-(g)_\infty\,=\, (dh)_0.\end{equation}
Therefore, the singular locus for the maxface is $\{p\,\in\, M\,\big\vert\,\, |g(p)|\,=\,1 \}$.

The completeness of a maxface can be determined using the following criterion:
\begin{fact}\label{fact}
A maxface is complete (see \cite[Corollary 4.8]{UMEHARA2006}) if and only if the following three hold:
\begin{enumerate}
\item $M$ is bi-holomorphic to $\overline{M}\setminus\{p_1,\, \cdots,\, p_n\}$, where $\overline{M}$ is a compact Riemann surface\color{black}.
\item $|g|\,\neq\, 1$ at $p_j$, $1\,\leq\, j\,\leq\, n$.
\item The induced metric $ds^2$ is complete at the ends.
\end{enumerate}
\end{fact}

\section{\textbf{Maxface with Enneper ends from a symmetric zigzag}}\label{section:Minimal 
surface and maxface with enneper end}

This section revisits the construction in \cite{weber1998minimal}
of a minimal surface using the zigzag method. Subsequently, 
we will determine the maxface that corresponds to this minimal surface as its companion. Furthermore, we will 
construct a maxface using the same zigzag approach and establish its relationship with the companion maxface.  

\subsection{Minimal surfaces by Weber and Wolf}\label{subsection:minimal surface weber & wolf}

\begin{definition}[{\cite{weber1998minimal}}]\label{definition:zigzag}
A zigzag $Z$ of genus $p$ is an open and properly embedded
arc in $\C$ composed of alternating horizontal and vertical sub-arcs with angles of $\frac{\pi}{2}, \frac{3\pi}{2},\frac{\pi}{2},\ldots,\frac{3\pi}{2},\frac{\pi}{2} $ between consecutive sides, and having
$2p + 1$ vertices (there are $2p + 2$ sides, including an initial infinite vertical side and a
terminal infinite horizontal side). A symmetric zigzag of genus $p$ is a zigzag
of genus $p$ which is symmetric about the line ${y \,=\, x}$. 
\end{definition}
\begin{figure}[h]
    \centering
    \includegraphics[scale=0.3]{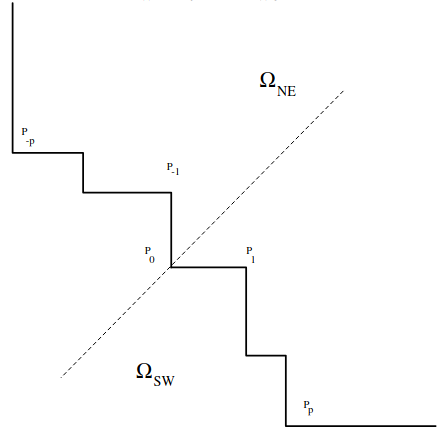}
    \caption{Zigzag of genus $p$ (Picture Courtesy: Weber and Wolf \cite{weber1998minimal})}
    \label{fig:zigzag}
\end{figure}
A symmetric zigzag of genus $p$ divides $\C$ into two regions, one of which we call by the name $\Omega_{NE}$
and the other by $\Omega_{SW}$ (see Figure \ref{fig:zigzag}).

\begin{definition}\label{definition:reflexive zigzag}
A symmetric zigzag $Z$ is called reflexive if there is a
conformal map $\phi\,:\, \Omega_{NE}(Z)\,\longrightarrow\, \Omega_{SW}(Z)$ which takes vertices to vertices.
\end{definition}

\subsubsection{Construction of minimal surface}

Let $Z$ be the genus $p$ reflexive zigzag separating $\C$ into two regions namely $\Omega_{NE}$ and $\Omega_{SW}$. 
We denote the vertices of $\Omega_{NE}$ by $\lbrace P_{j}\rbrace_{j=-p}^{p},\,P_{\infty}\,=\,\infty$. The vertices 
of $\Omega_{SW}$ are labeled in the reverse order: $Q_j\,=\,P_{-j}$ for $j\,\in\,\lbrace 
-p,\,-(p-1),\,\ldots,\,p\rbrace$ and $Q_{\infty}\,=\,\infty$. By doubling these two regions, we obtain two one-punctured spheres, denoted as $S_{NE}$ and $S_{SW}$, each with $2p+1$ marked points ${P_j}$ and ${Q_j}$, respectively, as well as a puncture at $P_{\infty}$ and $Q_{\infty}$ respectively.  
Further, we take the hyperelliptic cover $\mathcal{R}_{NE}$ (respectively, $\mathcal{R}_{SW}$)
of $S_{NE}$ (respectively, $S_{SW}$) branched at $\{P_j\}$ and (respectively, $\{Q_j\}$).
respectively. Let
\begin{equation}\label{epi}
\pi_{NE}\,:\, \mathcal{R}_{NE}\, \longrightarrow\, S_{NE}\ \ \,\text{ and }\ \ \,
\pi_{SW}\,:\, \mathcal{R}_{SW}\, \longrightarrow\, S_{SW}
\end{equation}
be the degree two branched covering maps.
Since the zigzag $Z$ is reflexive, there is a conformal map 
$\phi\,:\,\Omega_{NE}\,\longrightarrow\,\Omega_{SW}$ taking the vertices to vertices. This $\phi$ induces a 
conformal map 
$$\widetilde{\phi}\,:\,\mathcal{R}_{NE}\,\longrightarrow\,\mathcal{R}_{SW}$$ such that
$\widetilde{\phi}(P_j)\,=\,Q_j$ for all $j$.

The flat metric $|dz|$ on $\Omega_{NE}$ extends to a singular flat metric on $S_{NE}$ with cone angles at 
$P_j$. Its pullback through $\pi_{NE}$ (see \eqref{epi}) is a singular flat metric on $\mathcal{R}_{NE}$. 
The corresponding nonvanishing one form (see \cite{weber1998minimal}) is denoted by
$\omega_{NE}$. Similarly, denote by $\omega_{SW}$ the nonvanishing one form induced on $\mathcal{R}_{SW}$
by the flat metric $|dz|$ on $\Omega_{SW}$.

We define two flat forms $\alpha\,=\,\exp({\frac{-i \pi }{4}})\omega_{NE}$ and $\beta\,=\,
\exp({\frac{-i \pi }{4}})\widetilde{\phi}^*\omega_{SW}$ on $\mathcal{R}_{NE}$.
We choose $c$ and define $dh\,=\,c d\pi_{NE}$ such that $\alpha\beta\,=\,dh^2$ (see \cite{weber1998minimal}
for the details).

Finally, we define $g\,=\,\frac{\alpha}{dh}$ and
consider the formal Weierstrass data as $\lbrace g,\,dh\rbrace$. Weber and Wolf have proven in
\cite{weber1998minimal} that this pair gives the minimal surface by showing that
$$\int_{B_j} gdh\,=\,\int_{B_j}\alpha\,=\,\overline{\int_{B_j}g^{-1}dh}\,=\,\overline{\int_{B_j}\beta}$$
for all $\lbrace B_{\pm j}\rbrace_{j=1}^{p}$ (defined in the proof of \cite[Theorem 3.3]{weber1998minimal}). We
will explain this technique more explicitly in the next section, where we will modify it to generate
a maximal surface.

\subsection{Maximal surface generated from zigzag}\label{Subsec:Maximalsurface tildeX}

Using the terminology in Section \ref{subsection:minimal surface weber & wolf},
define two nonvanishing holomorphic one forms $\widetilde{\alpha}\,=\,e^{\frac{i \pi }{4}}\omega_{NE}$
and $\widetilde{\beta}\,=\,e^{\frac{i \pi }{4}}\widetilde{\phi}^*\omega_{SW}$. By definition, $B_j$
encloses exactly around the line segment $P_{j+1} P_{j}$ and $B_{-j}$ encloses exactly around the line segment $P_{-j-1}P_{-j}$. Since $\omega_{NE},\omega_{SW}$ are flat forms,
$$\int_{B_j}\widetilde{\alpha}\,=\, e^{\frac{i \pi }{4}}\int_{B_j}\omega_{NE}
\,=\,2e^{\frac{i \pi }{4}}\int_{P_{j+1}}^{P_{j}} dz\,=\,2e^{\frac{i \pi }{4}}(P_{j}-P_{j+1}).$$
Similarly, for $\widetilde\beta$,
$$
\int_{B_j}\widetilde{\beta}\,=\,e^{\frac{i \pi }{4}}\int_{\widetilde{\phi}(B_j)}\omega_{SW}
\,=2\color{black}\,e^{\frac{i \pi }{4}}\int_{Q_{j+1}}^{Q_{j}} dz
$$
$$
=\,2e^{\frac{i \pi }{4}}(Q_{j}-Q_{j+1})\,=\,2e^{\frac{i \pi }{4}}(P_{-j}-P_{-j-1}).
$$

By symmetry of zigzag, we have $-\overline{e^{\frac{i \pi }{4}}P_{-j}}
\,=\, e^{\frac{i \pi }{4}}P_{j}$.
Therefore, $-\overline{\int_{B_j}\widetilde{\beta}}\,=\,\int_{B_j}\widetilde{\alpha}$ for all $B_j$.

We define $\widetilde{g}\,=\,\frac{\widetilde{\alpha}}{\widetilde{dh}}$, where
$\widetilde{dh}\,=\,\widetilde{c}d\pi_{NE}$ for some $\widetilde{c}$, 
such that $\widetilde{\alpha}\widetilde{\beta}\,=\,\widetilde{dh}^2$. Therefore
$\widetilde{g}\widetilde{dh}\,=\,\widetilde{\alpha}$ and $\widetilde{g}^{-1}\widetilde{dh}
\,=\,\widetilde{\beta}$.  From the definition of $\widetilde g$ and $\widetilde dh$ the divisor condition is readily satisfied. 
Since $\widetilde{dh}$ is an exact form, we deduce that $\int_{B_j}\widetilde{dh}
\,=\,0$. Further from discussions in the last paragraph it is deduced that
$$\int_{B_j}\widetilde{g}\widetilde{dh}\,=\,-\overline{\int_{B_j}\widetilde{g}^{-1}\widetilde{dh}}$$
for all $B_j$. Further, since $\widetilde g$ has zero at $P_{\infty}$, the singularity set is compact. Thus we conclude the following:

\begin{theorem}\label{prop:maxfacefromZigzag}
The triple $\lbrace\mathcal{R}_{NE}\setminus\lbrace P_{\infty}\rbrace,\,\widetilde{g},\,\widetilde{dh}\rbrace$ defines a maxface.
This surface, we denote by $\widetilde{X}$, has at most eight symmetries. Moreover, the following
hold:
\begin{enumerate}
\item The minimal surface $X$ is not the companion of  $\widetilde{X}$.

\item The companion of the minimal surface $X$ exists; denote it by $X_C$. Then $X_C$ and
$\widetilde{X}$ are symmetric.
\end{enumerate}
\end{theorem}

\begin{proof}
It is easy to see that $\frac{\widetilde{\alpha}}{\alpha} \,=\, \frac{\widetilde{\beta}}{\beta} \,=\,
 i $. Therefore, $\frac{\widetilde{\alpha}\widetilde{\beta}}{\alpha\beta} \,=\, -1$. Thus,
    $$\left\{\frac{\widetilde{dh}^2}{dh^2} \,=\, -1\right\} \,\,\implies\,\, \{\widetilde{dh} \,=\,  i  dh\color{black}\}.$$
Therefore, $\widetilde{g} \,=\, \frac{ i \alpha}{ i  dh}
\,=\, \frac{\alpha}{dh} \,=\, g$. Thus, the corresponding maximal surface has data $(g,\,  i  dh)$.
The minimal surface and the maximal surface share the same underlying Riemann surface, and they have
at most eight conformal and anticonformal isometries.

Companion of $X$, which is denoted by $X_C$, is given by the Weierstrass data as $g_1 \,=\, -
 i  g$ and $dh_1 \,=\,  i  dh$, if the 
period condition holds.  We verify it here. Since $dh$ is exact, we have $\int_\gamma dh_1 \,=\, 0$ for all 
$\gamma \,\in\, \lbrace B_{\pm j} \rbrace_{j=1}^{p}$. As $\int_\gamma gdh \,=\, \overline{\int_\gamma g^{-1}dh}$, 
$$\int_{\gamma}g_1dh_1 \,= \,\int_\gamma gdh \,=\, \overline{\int_\gamma g^{-1}dh} \,=\, 
-\overline{\int_{\gamma}g_1^{-1}dh_1}.$$ Thus, $\lbrace g_1,\, dh_1 \rbrace$ is a Weierstrass 
data on $\mathcal{R}_{NE}$ for the maximal surface.

We have 
$$\widetilde{X}(p) \,=\, Re\int_{0}^p((g+\frac{1}{g})\frac{ i dh}{2},\,\frac{ i }{2}(g-
\frac{1}{g}){ i dh},\, i dh),$$
$$X_C(p) \,=\, Re\int_{0}^p((g-\frac{1}{g})\frac{dh}{2},\,\frac{ i }{2}(g+\frac{1}{g})dh,\, i dh).$$

It is straightforward to check that $\widetilde{X}(p) \,= \,\Psi(X_C(p))$ where $\Psi\,:\,
\mathbb{E}^3_1\,\longrightarrow\, \mathbb{E}^3_1$ is the map defined by $(x,\,y,\,z)
\,\mapsto\color{black}\, (y,\,-x,\,z)$. Therefore, these two maximal surfaces are symmetric.
\end{proof}

Similarly, one may attempt to construct minimal and maximal maps by zigzags symmetric about the line $y\,=\,-x$, 
but again, the surfaces turn out to be the same modulo symmetry.

\section{\textbf{Tweezers  and corresponding zigzags of genus $p$}}

In this section, we will first revisit the concept of an ``ortho-disk" as described in \cite{Weber1998TeichmullerTA}. 
We will then focus on a specific class of ortho-disks, which we refer to as ``tweezers". Although tweezers were 
inspired by ``zigzags", they differ significantly in many aspects. We will explore the 
relationships between the Riemann surface associated with tweezers and $\mathcal{R}_{NE}$ associated to
zigzags.

\subsection{Conformal polygon and ortho-disk \cite{Weber1998TeichmullerTA}}

On the upper half-plane, consider $n \,\geq\, 3$ marked points $\{t_j\}_{j=1}^n$ lying on the real 
line. The point $t_\infty = \infty$ is also treated as one of these marked points. The 
upper half-plane equipped with these marked points is
referred to as a conformal polygon, while the marked points are referred to as its vertices.
Two conformal polygons are called equivalent under conformal mapping if there exists a biholomorphism
of the upper half-plane that preserves the set of vertices while fixing the point $\infty$.

Let $a_j$, $j\, \in\, \{1,\, \cdots,\, n\}\cup\lbrace\color{black} \infty\rbrace\color{black}$, denote a set of odd integers such that
\begin{equation}\label{eqn:weight}
a_\infty \,\,=\,\, -4 - \sum_{j} a_j.
\end{equation}
A vertex $t_j$ is called ``finite'' if $a_j \,>\,-2$; otherwise, it is classified as ``infinite''.
According to \ref{eqn:weight}, there is at least one finite vertex, which may be coincide with $t_\infty$.
We have the corresponding Schwarz--Christoffel map
\[
F(z) \,:=\, \int_{ i }^z (t - t_1)^{\frac{a_1}{2}} \ldots (t - t_n)^{\frac{a_n}{2}} dt
\]
defined on the complement of the infinite vertices in the upper half-plane $\mathbb{H} \cup \mathbb{R}$.

Upper half plane, without the infinite vertices, equipped with the
pullback, by $F$, of the flat metric on $\mathbb{C}$ is called an ortho-disk.

The integer $a_j$ corresponds to cone angle of $\frac{a_j + 2}{2}\pi$ at $t_j$. Negative angles bear 
significance, as a vertex with a negative angle $-\theta$ resides at infinity and represents the 
intersection of two lines. These lines also intersect at a finite point, forming a positive angle
$+\theta$ at that intersection.

An ortho-disk is called symmetric if it has a reflectional symmetry which fixes two vertices.
Two ortho-disks that share the same underlying conformal polygon, but
having possibly different exponents, are called \textbf{conformal ortho-disks}. Consider two ortho-disks,
$X_1$ and $X_2$, each 
with distinct vertex data. These ortho-disks are termed \textbf{conjugate} if there exists a straight line $l 
\,\subset\, \mathbb{C}$ such that the corresponding periods are symmetric with respect to $l$.
Ortho-disks $X_1$  and $X_2$ are called \textbf{reflexive} if they are both conformal and conjugate.

Below, we will discuss a particular type of ortho-disk, which we call a tweezer. 

\subsection{Symmetric tweezers of genus $p$}\label{definition:tweezer}

A tweezer of genus $p$, with $p\,\geq\, 2$, is an open arc in $\C$ consisting of $2p+1$ vertices $\lbrace 
P_j\rbrace_{j=-p}^{p}$ and $2p+2$ edges such that
\begin{enumerate}
\item the interior angle of the region that is left when we go from $P_{p}$ to $P_{-p}$ alternates
between $\frac{3\pi}{2},\frac{\pi}{2}$ except at $P_{\pm 1},\,P_0$,

\item the interior angle at $P_{\pm 1}$ is always $\frac{\pi}{2}$, and

\item the interior angle at $P_0$ is $\frac{\pi}{2}$ when $p$ is even and it is $\frac{3\pi}{2}$ when $p$ is odd.
\end{enumerate}
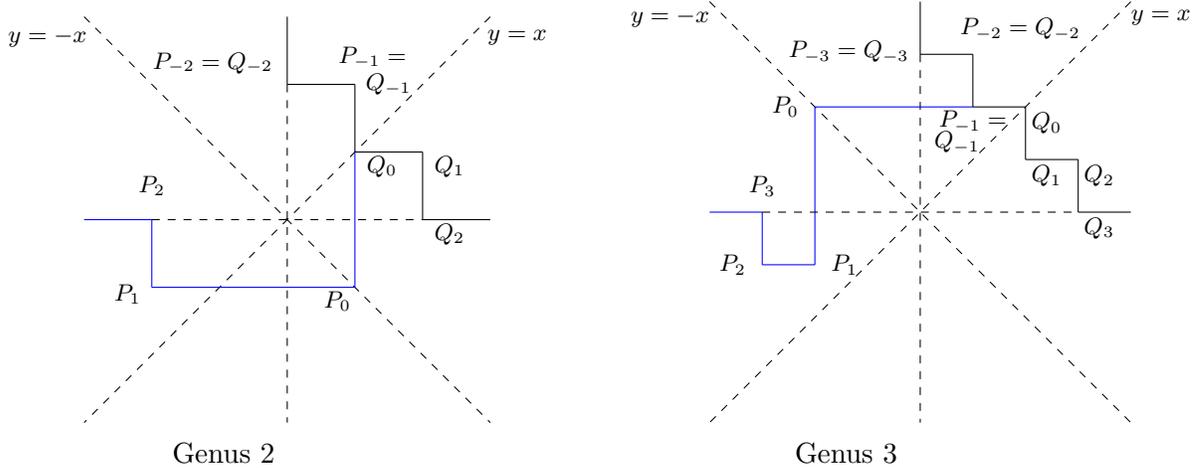
\begin{figure}[ht]
\begin{subfigure}{0.4\linewidth}
    \begin{tikzpicture}[scale=0.9]
   \draw[-,dashed] (3,3) -- (-3,-3);
   \draw (2.8,2.7)node[right] {\scriptsize$y=x$} ;
   \draw[-,dashed] (-3,3) -- (3,-3);
   \draw (-2.8,2.7)node[left] {\scriptsize$y=-x$} ;
   \draw[-] (3,0) -- (2,0);
   \draw (2.4,0.1)node[below] {\scriptsize$Q_{2}$};
   \draw[-] (2,0) -- (2,1);
   \draw (2.4,1.1)node[below] {\scriptsize$Q_{1}$};
   \draw[-] (2,1) -- (1,1);
   \draw (1.4,1.1)node[below] {\scriptsize$Q_{0}$};
   \draw[-] (1,1) -- (1,2);
   \draw (0.6,2.4)node[right] {\scriptsize$P_{-1}=$};
   \draw (1,2)node[right] {\scriptsize$Q_{-1}$};
   \draw[-] (1,2) -- (0,2);
   
   \draw[-] (0,2) -- (0,3);
   \draw (-1.1,2)node[above] {\scriptsize\scriptsize$P_{-2}=Q_{-2}$};
   \draw[-,blue] (1,1) -- (1,-1);
   \draw (1.1,-1.2)node[left]{\scriptsize $P_0$};
   \draw[-,blue] (1,-1) -- (-2,-1);
   \draw (-2,-1.1)node[left]{\scriptsize $P_1$};
   \draw[-,blue] (-2,-1) -- (-2,0);
   \draw (-2,0.2)node[above]{\scriptsize $P_2$};
   \draw[-,blue] (-2,0) -- (-3,0);
   \draw[-,dashed] (0,2) -- (0,-3);
   \draw[-,dashed] (-2,0) -- (2,0);
   
   \end{tikzpicture}
   \subcaption*{Genus 2}
\end{subfigure}
\hspace{2cm}
   \begin{subfigure}{0.4\linewidth}
    \begin{tikzpicture}[scale=0.7]
   \draw[-,dashed] (4,4) -- (-4,-4);
   \draw (3.8,3.7)node[right] {\scriptsize$y=x$} ;
   \draw[-,dashed] (-4,4) -- (4,-4);
   \draw (-3.8,3.7)node[left] {\scriptsize$y=-x$} ;
   \draw[-] (4,0) -- (3,0);
   \draw (3.4,0.1)node[below] {\scriptsize$Q_{3}$};
   \draw[-] (3,0) -- (3,1);
   \draw (3.4,1.1)node[below] {\scriptsize$Q_{2}^{}$};
   \draw[-] (3,1) -- (2,1);
   \draw (2.4,1.1)node[below] {\scriptsize$Q_{1}^{}$};
   \draw[-] (2,1) -- (2,2);
   \draw (2.4,2.1)node[below] {\scriptsize$Q_{0}^{}$};
   \draw[-] (2,2) -- (1,2);
   \draw (1,1.3)node[above] {\scriptsize$P_{-1}^{}=$};
   \draw (0.7,0.9)node[above]{\scriptsize$Q_{-1}^{}$} ;
   \draw[-] (1,2) -- (1,3);
   \draw (1.9,3.05)node[above] {\scriptsize\scriptsize$P_{-2}^{}=Q_{-2}^{}$};
   \draw[-] (1,3) -- (0,3);
   \draw (-0.01,3.05)node[left] {\scriptsize\scriptsize$P_{-3}^{}=Q_{-3}^{}$};
   \draw[-] (0,3) -- (0,4);
   \draw[-,dashed] (0,3) -- (0,-4);
   \draw[-,blue] (1,2) -- (-2,2);
   \draw (-2.1,2)node[left] {\scriptsize$P_0^{}$};
   \draw[-,blue] (-2,2) -- (-2,-1);
   \draw (-1.9,-1)node[right] {\scriptsize$P_1^{}$};
   \draw [-,blue] (-2,-1) -- (-3,-1);
   \draw (-3.1,-1)node[left] {\scriptsize$P_2^{}$};
   \draw [-,blue] (-3,-1) -- (-3,0);
   \draw (-3,0.1)node[above] {\scriptsize$P_3^{}$};
   \draw [-,blue] (-3,0) -- (-4,0);
   \draw[-,dashed] (-3,0) -- (3,0);
   \end{tikzpicture}
   \subcaption*{Genus 3}
\end{subfigure}
\caption{Symmetric tweezers and their corresponding zigzags of genus two and genus three}\label{figure 3:different tweezers g2 g3}
\end{figure}
\begin{figure}[ht]
\begin{subfigure}{0.48\linewidth}
\centering
    \begin{tikzpicture}[scale=0.7]
   \draw[-,dashed] (5,5) -- (-5,-5);
   \draw (4.1,4)node[right] {\scriptsize$y=x$} ;
   \draw[-,dashed] (-5,5) -- (5,-5);
   \draw (-4.8,4.7)node[left] {\scriptsize$y=-x$} ;
   \draw[-] (5,0) -- (4,0);
   \draw (4.4,0.1)node[below] {\scriptsize$Q_{4}^{}$};
   \draw[-] (4,0) -- (4,1);
   \draw (4.4,1.1)node[below] {\scriptsize$Q_{3}^{}$};
   \draw[-] (4,1) -- (3,1);
   \draw (2.8,1.1)node[below] {\scriptsize$Q_{2}^{}$};
   \draw[-] (3,1) -- (3,2);
   \draw (3.1,1.8)node[right] {\scriptsize$Q_{1}^{}$};
   \draw[-] (3,2) -- (2,2);
   \draw (1.9,1.7)node[right] {\scriptsize$Q_{0}^{}$};
   
   \draw[-] (2,2) -- (2,3);
   \draw (2.1,3.3)node[above] {\scriptsize\scriptsize$P_{-1}^{}=$};
   \draw (2.1,2.9)node[above] {\scriptsize\scriptsize$Q_{-1}^{}$};
   \draw[-] (2,3) -- (1,3);
   \draw (0.9,2.1)node[above] {\scriptsize\scriptsize$P_{-2}^{}=$};
   \draw (0.9,1.7)node[above] {\scriptsize\scriptsize$Q_{-2}^{}$};
   \draw[-] (1,3) -- (1,4);
   \draw (1,4.3)node[above] {\scriptsize\scriptsize$P_{-3}^{}=$};
   \draw (1,3.9)node[above] {\scriptsize\scriptsize$Q_{-3}^{}$};
   \draw[-] (1,4) -- (0,4);
   \draw (-1.7,4)node[right] {\scriptsize\scriptsize$P_{-4}^{}=$};
   \draw (-1.7,3.5)node[right] {\scriptsize\scriptsize$Q_{-4}^{}$};
   \draw[-] (0,4) -- (0,5);
   \draw[-,blue] (2,2) -- (2,-2);
   \draw (2.15,-2.3)node[left]{\scriptsize $P_0^{}$};
   \draw[-,blue] (2,-2) -- (-3,-2);
   \draw (-3,-2.1)node[left]{\scriptsize $P_1^{}$};
   \draw[-,blue] (-3,-2) -- (-3,-1);
   \draw (-2.7,-1.1)node[above]{\scriptsize $P_2^{}$};
   \draw[-,blue] (-3,-1) -- (-4,-1);
   \draw (-4.6,-1.1)node[right]{\scriptsize $P_3^{}$};
   \draw[-,blue] (-4,-1) -- (-4,0);
   \draw (-4,0.7)node[below]{\scriptsize $P_4^{}$};
   \draw[-,blue] (-4,0) -- (-5,0);
   \draw[-,dashed] (0,4) -- (0,-5);
   \draw[-,dashed] (-4,0) -- (4,0);
   
   \end{tikzpicture}
   \subcaption*{Genus 4}
\end{subfigure}
\hspace{1cm}
   \begin{subfigure}{0.4\linewidth}
   \centering
    \begin{tikzpicture}[scale=0.6]
   \draw[-,dashed] (6,6) -- (-6,-6);
   \draw (5.8,5.7)node[right] {\scriptsize$y=x$} ;
   \draw[-,dashed] (-6,6) -- (6,-6);
   \draw (-2.8,5.7)node[left] {\scriptsize$y=-x$} ;
   \draw[-] (6,0) -- (5,0);
   \draw (5.4,0.1)node[below] {\scriptsize$Q_{5}^{}$};
   \draw[-] (5,0) -- (5,1);
   \draw (5.4,1.1)node[below] {\scriptsize$Q_{4}^{}$};
   \draw[-] (5,1) -- (4,1);
   \draw (4.4,1.1)node[below] {\scriptsize$Q_{3}^{}$};
   \draw[-] (4,1) -- (4,2);
   \draw (4.4,2.1)node[below] {\scriptsize$Q_{2}^{}$};
   \draw[-] (4,2) -- (3,2);
   \draw (2.8,1.3)node[above] {\scriptsize$Q_{1}^{}$};
   
   \draw[-] (3,2) -- (3,3);
   \draw (3.5,3)node[right]{\scriptsize$Q_{0}^{}$} ;

   \draw[-] (3,3) -- (2,3);
      \draw (1.6,2.2)node[above] {\scriptsize\scriptsize$P_{-1}^{}=$};
   \draw (1.4,1.7)node[above] {\scriptsize\scriptsize$Q_{-1}^{}$};
   \draw[-] (2,3) -- (2,4);
   \draw (1.7,4.3)node[right]{\scriptsize$P_{-2}^{}=$};
   \draw (1.9,3.8)node[right]{\scriptsize$Q_{-2}^{}$};
   \draw[-] (2,4) -- (1,4);
   \draw(0.5,3.6)node[above]{\scriptsize$P_{-3}^{}$};
   \draw(0.9,3.5)node{\scriptsize$=Q_{-3}^{}$};
   \draw[-] (1,4) -- (1,5);
   \draw(1.1,5.2)node[right]{\scriptsize$P_{-4}^{}=Q_{-4}^{}$};
   \draw[-] (1,5) -- (0,5);
   \draw(-1.9,5)node[right]{\scriptsize$P_{-5}^{}=$};
   \draw(-1.7,4.4)node[right]{\scriptsize$Q_{-5}^{}$};
   \draw[-] (0,5) -- (0,6);
   \draw[-,dashed] (0,5) -- (0,-6);
   \draw (-2.1,2)node[left] {\scriptsize$P_0^{}$};
   \draw[-,blue] (2,3) -- (-3,3);
   \draw (-3,-1.7)node[right] {\scriptsize$P_1^{}$};
   \draw [-,blue] (-3,3) -- (-3,-2);
   \draw (-3.8,-1.9)node[left] {\scriptsize$P_2^{}$};
   \draw [-,blue] (-3,-2) -- (-4,-2);
   \draw (-3.8,-1.2)node[above] {\scriptsize$P_3^{}$};
   \draw [-,blue] (-4,-2) -- (-4,-1);
   \draw (-5.5,-1.2)node[above] {\scriptsize$P_4^{}$};
   \draw [-,blue] (-4,-1) -- (-5,-1);
   \draw(-5,0.3)node[above]{\scriptsize$P_{5}^{}$};
   \draw [-,blue] (-5,-1) -- (-5,0);
   \draw[-,blue] (-5,0) -- (-6,0);
   \draw[-,dashed] (-5,0) -- (5,0);
   \end{tikzpicture}
   \subcaption*{Genus 5}
\end{subfigure}
\caption{Symmetric tweezers and their corresponding zigzags of genus four and genus five}\label{figure 4:different tweezers g4 g5}
\end{figure}
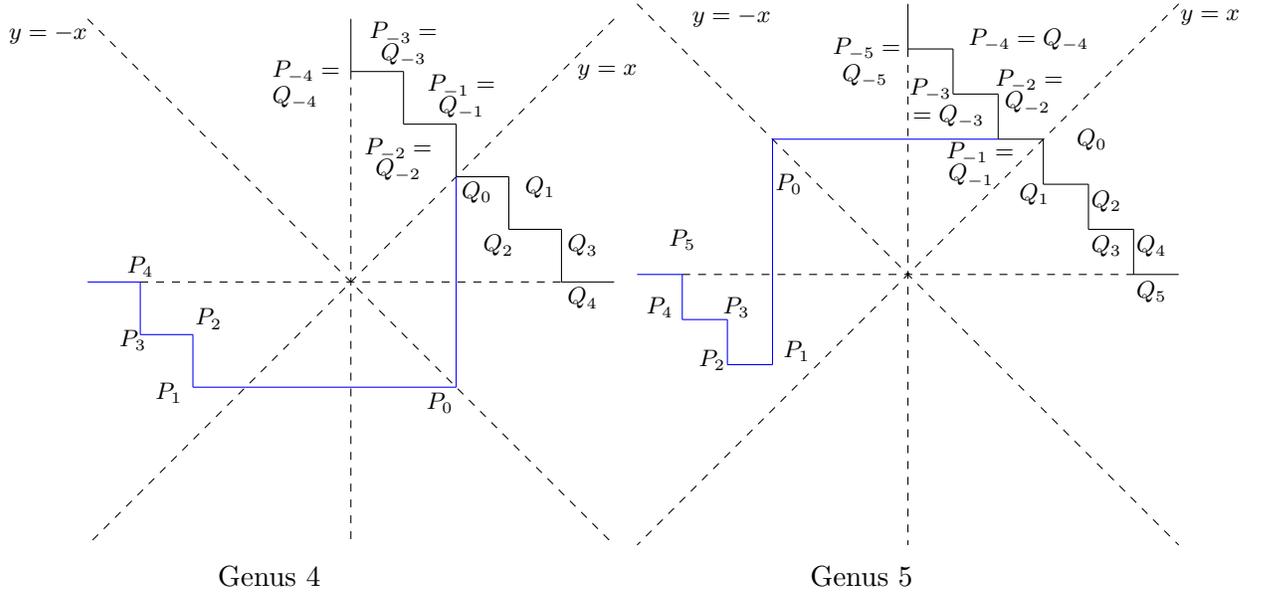
Moreover, at each vertex, there is an exterior angle assigned to it. If at $P_k$ the interior angle as above 
is $\theta$, then the corresponding exterior angle is $2\pi-\theta$. The notion of interior and exterior 
angles help us to recognize the image of $\H$ under the Schwarz--Christoffel map, as discussed below. For two 
fixed set of real numbers $t_{-p}\,<\, \ldots\,<\,t_p$, we will use the notation $\Omega_{Gdh}$ (following 
\cite{Weber1998TeichmullerTA}) to mean $\H\cup\R$ with the flat metric induced by the unique Schwarz--Christoffel 
map taking $\H$ to the interior of the polygon enclosed by the tweezer in the side of the interior angles 
such that the map sends $\R$ to the tweezer and the point $t_j$ to $P_j$. Similarly,  for another set of real numbers $s_{-p}\,<\,\ldots\,<\,s_p$ we can find a unique 
Schwarz--Christoffel map, sending $\H$ to the interior polygon enclosed by the tweezer in the side of the 
exterior angle such that the image of $s_j$ is $P_{-j}$ which we call as $Q_j$.  We call $\H\cup\R$ with the flat metric induced by 
this map as $\Omega_{G^{-1}dh}$. We write the Schwarz--Christoffel maps for the corresponding regions:
\begin{eqnarray*}
    z\,\,\longmapsto\,\,
    \int_{ i }^z \prod_{j=-p}^p (t-t_j)^{\frac{a_j}{2}}dt & \text{ on }\ \, \Omega_{Gdh},\\
    z\,\,\longmapsto\,\,
    \int_{ i }^z \prod_{j=-p}^p (t-s_j)^{\frac{b_j}{2}}dt & \text{ on }\ \, \Omega_{G^{-1}dh},
\end{eqnarray*}
where
\begin{eqnarray*}
    a_j\,=\,\begin{cases}
    \pm 1 & \text{ alternatively when }\, j \,\neq\, \pm 1,\,0,\\
    -1 & \text{ when }\, j\,=\,\pm 1,\\
    1  & \text{ when }\, j\,=\,0,\,p\, \text{ is odd},\\
    -1  & \text{ when }\, j\,=\,0,\,p\, \text{ is even},
\end{cases}\\
b_j\,=\,\begin{cases}
    \mp 1 & \text{ alternatively when }\, j \,\neq\, \pm 1,\,0,\\
    1 & \text{ when }\, j\,=\,\pm 1,\\
    -1  & \text{ when }\, j\,=\,0,\,p\, \text{ is odd},\\
    1  & \text{ when }\, j\,=0,\,p\, \text{ is even}.
\end{cases}
\end{eqnarray*}

It is evident from our notation that the same convention used for naming vertices in the case of a zigzag has been applied to the vertices of the boundary of $\Omega_{Gdh}$ and $\Omega_{G^{-1}dh}$. If we denote the vertices of $\Omega_{Gdh}$ as $\lbrace P_j\rbrace_{j=-p}^p$ with $P_{\infty}=\infty$, then we have chosen to name the vertices in the reverse order for $\Omega_{G^{-1}dh}$, i.e., the vertices are labeled as $Q_j=P_{-j}$ and $Q_{\infty}=\infty$

From a tweezer, we can obtain a zigzag by simply mapping the points of the tweezers $P_j$, where $j\,\geq\, 1$, 
to $-P_j$,  $P_0$  to $- i P_0$ when $p$ is even and $P_0$ to $ i P_0$ when $p$ is odd. Reversing the same process, we can get the corresponding tweezer from a given zigzag. For the case $p=1$, the corresponding tweezer of genus $1$ zigzag is just the rotation of the zigzag. Therefore, we skip the discussion of genus $1$ tweezer except in some places where we need it. This zigzag may not necessarily be symmetric.

We call a tweezer symmetric if it is symmetric with respect to the line $y\,=\,-x$. It is clear that
a zigzag that corresponds to a symmetric tweezer is symmetric along the line $y\,=\,x$.

\begin{definition}\label{defn:reflxive tweezers}
A symmetric tweezer is said to be reflexive if there is a conformal map $\phi\,:\, \Omega_{Gdh}
\,\longrightarrow\, \Omega_{G^{-1}dh}$ taking vertices to vertices, i.e., $\phi(P_j)\,=\,Q_j$.
\end{definition}

In the context of an ortho-disk, a tweezer gives rise to two ortho-disks, denoted as $\Omega_{Gdh}$ with vertex data $\lbrace t_j\rbrace_{j=-p}^{p}$ and $\Omega_{G^{-1}dh}$ with vertex data $\lbrace s_j\rbrace_{j=-p}^{p}$. Here, $t_j$ is mapped to $P_j$, and $s_j$ is mapped to $Q_j$. For a symmetric tweezer, these two ortho-disks are conjugate. Reflexivity implies that the corresponding conformal polygons are conformal.

\subsection{Riemann surfaces from the tweezer}

Similar to zigzags as in Section \ref{subsection:minimal surface weber & wolf}, we can construct first the one punctured spheres $S_{\Omega_{Gdh}},S_{\Omega_{G^{-1}dh}}$ with marked points $\lbrace P_j\rbrace_{j=-p}^{p}, \lbrace P^{\prime}_{j}= P_{-j}\rbrace_{j=-p}^{p}$ with puncture $\{P_{\infty}=\infty,\;P^{\prime}_{\infty}=\infty\}$,   and finally the hyperelliptic Riemann surfaces $\mathcal{R}_{Gdh},\;\mathcal{R}_{G^{-1}dh}$ respectively, and if the tweezer is reflexive, we get Riemann surfaces $\mathcal{R}_{Gdh},\; \mathcal{R}_{G^{-1}dh}$ conformal taking corresponding Weierstrass points to Weierstrass points.   We denote these Riemann surfaces as $R_T$.   For a given tweezer, we can construct the zigzag as discussed in the earlier subsection. Let for the corresponding zigzag, the marked sphere be $S_Z$ with puncture $\{P_{\infty}=\color{black}\infty\}$\color{black}, and the corresponding hyperelliptic Riemann surfaces 
be $R_Z$.

\begin{prop}\label{prof:RiemannsurfaceComprasion}
   For a fixed genus $p \geq 2$, the Riemann surface $R_T$ is neither conformal nor anticonformal to $R_Z$ by a mapping that maps Weierstrass points to Weierstrass points. Here, $T$ and $Z$ denote the tweezer and its corresponding zigzag of genus $p$, respectively.

\end{prop}
\begin{proof}
Suppose there is a conformal or anticonformal map $f$ between $R_T$ and $R_Z$ that takes corresponding Weierstrass points to each other, i.e., $f(W_T) = W_Z$ for Weierstrass points $W_T$ on $R_T$ and $W_Z$ on $R_Z$. This means the marked sphere $S_T$ maps to the marked sphere $S_Z$ with $f(P_j)=Q_j$, $f(P_{\infty})= Q_\infty$.

If $f$ is conformal and it fixes at least $4$ points on the real line, which are the Weierstrass points when we consider the restriction of $f$ to the sphere $S_{T}$. In this case, therefore, $f$ would fix the entire real line, effectively taking the tweezer $T$ to the corresponding zigzag $Z$.

For the zigzag,  the angle between two consecutive sides alternates, whereas for the tweezer when $p\geq 2$, this alternation will not always occur (as seen in the Figure \ref{figure 3:different tweezers g2 g3}, \ref{figure 3:different tweezers g4 g5}).

If $p$ is even, taking a neighborhood $\mathcal{U}$ containing the arc $P_{-2}P_{-1}P_{0}P_1P_2$, we would find that the restricted map  $f$  on $\mathcal{U}$ should map this arc to the corresponding arc of the zigzag $Q_{-2}Q_{-1}Q_0Q_{1}Q_2$. However, the interior angles at $P_{-1},P_{0},P_{1}$ are $\frac{\pi}{2},\frac{\pi}{2},\frac{\pi}{2}$ respectively, while the corresponding angles at $Q_{-1},Q_{0},Q_{1}$ are $\frac{3\pi}{2},\frac{\pi}{2},\frac{3\pi}{2}$ respectively, which is impossible due to conformality. Thus, no such conformal $f$ exists when $p$ is even and $p\geq 2$.

For odd $p\geq 2$, considering a neighborhood $\mathcal{V}$  containing the arc $P_{-3}P_{-2}P_{-1}P_0$, a similar comparison between the angles at $P_{-2},P_{-1}$ and their corresponding vertices in $Z$ leads to a contradiction. Hence, no such $f$ exists when $p$ is odd.

If there is an anticonformal map with similar conditions, the arguments remain the same, concluding that no such anticonformal map exists for either even or odd $p\geq 2$.
\end{proof}

In view of the above proposition, if we construct a maxface and minimal surface, those will not be isometric to the ones we get from the zigzags as in Section \ref{section:Minimal surface and maxface with enneper end}. In the next section, we will generate minimal and maximal surfaces using these tweezers, similar to the previous subsections. 

\section{\textbf{Minimal Surface and Maximal Surface with Tweezers $X_T$ and $\widetilde{X}_T$}}

Similar to the zigzag case as in the Section \ref{subsection:minimal surface weber & wolf}, we obtain non-vanishing holomorphic forms $\omega_{Gdh}$ and $\omega_{G^{-1}dh}$ on $\mathcal{R}_{Gdh}$ and $\mathcal{R}_{G^{-1}dh}$ respectively.
If we 
start from the reflexive tweezer $T$, the conformal map $\phi\,:\, \Omega_{Gdh} \,\longrightarrow\, 
\Omega_{G^{-1}dh}$ can be extended to a map $\widetilde\phi\,:\, 
\mathcal{R}_{Gdh} \,\longrightarrow\, \mathcal{R}_{G^{-1}dh}$ --- between
the corresponding hyperelliptic Riemann surfaces --- for which $\widetilde\phi(P_j) \,=\, Q_j$ for all $j$.
Define the following four holomorphic forms on $\mathcal{R}_{Gdh}$:
\begin{align*}
    \alpha_T &\,=\,
e^{-\frac{i \pi }{4}}\omega_{Gdh}, \\
    \beta_T &\,=\,
e^{-\frac{i \pi }{4}}\widetilde{\phi}^*(\omega_{G^{-1}dh}), \\
    \widetilde\alpha_T &\,= \,
e^{\frac{i \pi }{4}}\omega_{Gdh}, \\
    \widetilde\beta_T &\,= \,
e^{\frac{i \pi }{4}}\widetilde{\phi}^*(\omega_{G^{-1}dh}).
\end{align*}

Using the relationship between the cone angles and the order of the zeros of the $1$-forms, we find the divisors 
corresponding to these forms:

\begin{align*}
    (\alpha_T) = (\widetilde\alpha_T) &=\,
    \begin{cases}
        P_{\pm p}^2 P_{\pm (p-2)}^2 \ldots P^2_{\pm 3}P^2_{0}P_{\infty}^{-2} & \text{if } p\color{black} \text{ is odd}\\
        P_{\pm p}^2 P_{\pm (p-2)}^2 \ldots P^2_{\pm 2}P_{\infty}^{-2} & \text{if } p\color{black} \text{ is even}
    \end{cases}\\
    (\beta_T) = (\widetilde\beta_T) &=\,
    \begin{cases}
        P_{\pm (p-1)}^2 P_{\pm (p-3)}^2 \ldots P^2_{\pm 2}P_{\pm 1}^2P_{\infty}^{-4} & \text{if } p\color{black} \text{ is odd}\\
        P_{\pm (p-1)}^2 P_{\pm (p-3)}^2 \ldots P_{\pm 3}^2P_{\pm 1}^2P_{0}^2P_{\infty}^{-4} & \text{if }  p\color{black} \text{ is even}
    \end{cases}
\end{align*}

It is important to highlight that we use multiplicative notation for the divisor, while addition and subtraction follow the conventional complex number operations.   

The differential of the holomorphic covering $\pi_{Gdh}$ (branched at $P_j$) has the following 
divisor:
\begin{equation*}
    (d\pi_{Gdh}) \,=\, P_{\pm p}^1P_{\pm (p-1)}^1 \ldots P_0^1 P_{\infty}^{-3}.
\end{equation*}
The quadratic differentials $\alpha_T \beta_T$ and $d\pi_{Gdh}^2$ share an identical set of zeros and poles, as do $\widetilde{\alpha_T}\widetilde{\beta_T}$. Consequently, there exist appropriate constants $c$ and $\widetilde{c_T}$ such that: 
$\alpha_T\beta_T = (dh_T)^2$, where $dh_T = c_T d\pi_{Gdh}$ and 
$\widetilde{\alpha_T}\widetilde{\beta_T} = (\widetilde{dh_T})^2$, where $\widetilde{dh_T} = \widetilde{c_T}d\pi_{Gdh}$

Now, we define the following formal Weierstrass data for our minimal surface and maximal maps:

\begin{align*}
    \left(G_T \,=\, \frac{\alpha_T}{dh_T},\, dh_T\right)\, \text{ for the minimal surface,}\\
   \left(\widetilde{G}_T \,=\, \frac{\widetilde\alpha_T}{\widetilde{dh_T}},\, \widetilde{dh_T}\right)\, \text{ for the maximal surface}.
\end{align*}

Note that the divisor condition, as in the Equation \ref{divisor}, is trivially satisfied since at each zero and pole of $G_T$ and $\widetilde{G}_T$, there exists a zero of $d\pi_{Gdh}$ with an equal order.

The period problem can be resolved using the same technique as discussed in Section \ref{Subsec:Maximalsurface tildeX}. We begin with the basis of homology, denoted as $H_1(\mathcal{R}_{{Gdh}}, \mathbb Z)$.   For $j = -p, \ldots, p - 1$, let $B_j$ represent the loop in $S_{\Omega G_{dh}}$ that encloses only the line segment $P_jP_{j+1}$ within the disk and no other vertices in its interior.  These curves have closed lifts $\widetilde{B_j}$ to $\mathcal{R}_{Gdh}$ and form a homology basis of $\mathcal{R}_{Gdh}$. The following statements are valid:

\begin{align*}
    \int_{\widetilde{B_j}}\alpha_T &= 2e^{-\frac{i \pi }{4}}(P_j-P_{j+1}) \\
    \int_{\widetilde{B_j}}\beta_T &= 2e^{-\frac{i \pi }{4}}(P_{-j}-P_{-j-1}) \\
    \int_{\widetilde{B_j}}\widetilde\alpha_T &= 2e^{\frac{i \pi }{4}}(P_j-P_{j+1}) \\
    \int_{\widetilde{B_j}}\widetilde\beta_T &= 2e^{\frac{i \pi }{4}}(P_{-j}-P_{-j-1}).
\end{align*}

Due to the symmetry of the tweezers, we can further deduce that

\begin{align}
    \int_{\widetilde{B_j}}\alpha_T &\,=\, \overline{\int_{\widetilde{B_j}}\beta_T} \label{equation 8.1:period problem:minimal surface} \\
    \int_{\widetilde{B_j}}\widetilde\alpha_T &\,=\, -\overline{\int_{\widetilde{B_j}}\widetilde\beta_T}. \label{equation 8.2:period problem:maximal surface}
\end{align}

Moreover $dh_T$ and $\widetilde{dh_T}$ are exact, therefore for all loops $\int_{\widetilde{B_j}}dh_T\,=\,\int_{\widetilde{B_{j}}}\widetilde{dh_T}\,=\,0$. Thus, Equations \eqref{equation 
8.1:period problem:minimal surface} and \eqref{equation 8.2:period problem:maximal surface} confirm that the 
following maps are minimal and maximal, respectively:

\begin{align}
    X_T(z\color{black}) &\,=\, \operatorname{Re} \int_0^{z\color{black}} \left( (G_T^{-1} - G_T)  \frac{dh_T}{2},\,
  i  ( G_T^{-1} + G_T ) \frac{dh_T}{2},\, dh_T \right) \label{eqn:minimalsurfaceTweezer} \\
    \widetilde{X}_T(z\color{black}) &\,=\, \operatorname{Re} \int_0^{z\color{black}} \left( (\widetilde{G}_T^{-1} + G_T)
\frac{\widetilde{dh_T}}{2},\,   i  ( \widetilde{G}_T^{-1} -\widetilde{G}_T ) \frac{\widetilde{dh_T}}{2},\,  \widetilde{dh_T} \right).\label{eqn:maximalsurfacetweezer}
\end{align}

We have the following theorem\color{blue}:\color{black} 
\begin{theorem}\label{thm:surfaceFromTweezer}
    Given a reflexive tweezer $T$ of genus $p$, there exist a minimal surface $X_T$ and a maxface $\widetilde{X}_T$ of genus $p$, each having one Enneper end and at most eight symmetries. Furthermore, $X_T$ (respectively, $\widetilde{X}_T$) is not symmetric to the minimal surface $X$ (respectively, $\widetilde{X}$) as discussed in Section \ref{section:Minimal surface and maxface with enneper end}.
\end{theorem}

\begin{proof}
Equation \eqref{eqn:minimalsurfaceTweezer} defines the minimal surface, and Equation \eqref{eqn:maximalsurfacetweezer} defines the corresponding maximal map.

As we have already established the divisor condition, as in Equation \eqref{divisor}, it follows that the maximal map given in Equation \eqref{eqn:maximalsurfacetweezer} indeed represents a maxface.

Concerning completeness, we first observe that at the end of the maxface, the Gauss map has a zero. Consequently, the singular set  $\{p\color{black}: |\widetilde{G}_T(p\color{black})|=1\}$ is compact. Furthermore, we can confirm that the metric is complete at the end of both $X_T$ and $\widetilde{X}_T$ since the data at the end is the same as the data for the Enneper end.

The remaining task involves proving that the minimal surface $X_T$ is not symmetric to the minimal surface $X$. If such symmetry existed, there would either be a conformal or anticonformal diffeomorphism between the corresponding Riemann surfaces. However, as stated in Proposition \ref{prof:RiemannsurfaceComprasion}, this is not possible.
\end{proof}

\section{\textbf{Existence of a reflexive tweezer}}

In this section, we will demonstrate the existence of a reflexive tweezer. The proof closely follows the 
approach used for the zigzag case, as detailed in Section 5 of \cite{weber1998minimal}. While we could have 
directly stated the existence of the tweezer as a corollary of the zigzag and ortho-disks case, we choose to 
present it here for the sake of clarity and comprehensiveness.
Therefore, the content presented below does not introduce new concepts; instead, it serves as an application of the arguments found in various works by Weber and Wolf. Consequently, we aim to emphasize several key points within the context of tweezers, and we will discuss these in the following subsections.

\subsection{Space of Tweezers $\mathcal{T}_p$}

Two symmetric tweezers of genus $p$, denoted as $T$ and $T'$, are considered equivalent if the corresponding pairs 
of regions $(\Omega^T_{Gdh},\, \Omega^T_{G^{-1}dh})$ and $(\Omega^{T'}_{Gdh},\, 
\Omega^{T'}_{G^{-1}dh})$ are conformal by a map that takes each vertex to the same vertex).
Let $\mathcal{T}_p$ denote the class of equivalent symmetric tweezers of genus $p$. Furthermore, there always 
exists a conformal map of $\mathbb{C}$ that carries any symmetric tweezer to a symmetric tweezer $t_0$, such that 
the two endpoints $P_{-p}$ and $P_p$ are $i$ and $-1$ respectively, and $P_{k}=-i\overline{P_{-k}}$. From now 
onwards, we will represent a class of tweezer $[T]\,\in\,\mathcal{T}_p$ by the corresponding tweezer $t_0$ in $[T]$, 
unless stated otherwise.

Similar to the space $\mathcal{Z}_p$ as defined in \cite{weber1998minimal}, we define a map 
$$\mathcal{T}_p\,:\,\mathcal{T}_p\,\longrightarrow\,\mathbb{C}^{p-1}$$ as follows:
\[
\mathcal{T}_p([t_0]) \,=\, (P_1(t_0),\,\cdots,\,P_{p-1}(t_0)).
\]

This map induces a topology on $\mathcal{T}_p$, making it homeomorphic to a $(p-1)$-cell.

\subsection{Height function, sufficient condition of the reflexive tweezers, and its properness}

We will follow Subsection 4.2 of \cite{weber1998minimal} in the present context.  Consider the $2(p+1)$-pointed
sphere $S_{\Omega_{Gdh}}$ with marked points $P_{-p},\, \ldots,\, P_0,\, \ldots,\, P_p$, and $P_\infty$. Remarkably, 
$S_{\Omega_{Gdh}}$ exhibits two distinct reflective symmetries: one related to the image of $T$, and another 
arising from a corresponding symmetry of the tweezer.

We denote $[C_k]$ the homotopy class of simple curves that encircle the points $P_k$ and $P_{k+1}$ for
$k \,=\, 1,\, \ldots,\, p-1$. Similarly, $[C_{-k}]$ is the homotopy class of simple curves enclosing the
points $P_{-k}$ and $P_{-k-1}$ for $k \,=\, 1,\, \ldots,\, p - 1$. Furthermore, we define $[\alpha_k]$ as the combined pair of classes $[C_k] \cup [C_{-k}]$.

From the homotopy class of mappings that connect $S_{\Omega_{Gdh}}$ to $S_{\Omega_{G^{-1}dh}}$ (while preserving 
each of the vertices), we derive corresponding homotopy classes of curves on $S_{\Omega_{G^{-1}dh}}$, which are also
denoted by $[\alpha_k]$.
Additionally, we denote $$\text{E}_{\Omega_{Gdh}}(k)\,:=\, \text{Ext}_{S_{\Omega_{Gdh}}}([\alpha_k])
\ \  \text{ and } \ \ \text{E}_{\Omega_{G^{-1}dh}}(k)\,:=
\,\text{Ext}_{S_{\Omega_{G^{-1}dh}}}([\alpha_k]),$$ representing the extremal lengths of $[\alpha_k]$
within the domains of $S_{\Omega_{Gdh}}$ and $S_{\Omega_{G^{-1}dh}}$ respectively. 

Motivated by Definition 4.4 of \cite{weber1998minimal}, we construct the similar height function
$$D^T\,:\,\mathcal{T}_p \,\longrightarrow\, \mathbb{R}$$ defined by

\begin{eqnarray}\label{defn: DT}
D^T(T) = &&\sum_{j=1}^{p-1} \left(\exp\left(\frac{1}{\text{E}_{\Omega_{Gdh}}(j)}\right) - \exp\left(\frac{1}{\text{E}_{\Omega_{G^{-1}dh}}(j)}\right)\right)^2 \nonumber\\
&&+  \sum_{j=1}^{p-1}\left(\text{E}_{\Omega_{Gdh}}(j) - \text{E}_{\Omega_{G^{-1}dh}}(j)\right)^{2}.
\end{eqnarray}
From the same argument as in the case of a zigzag (see \cite[Section 4.3]{weber1998minimal}) it is clear that
$D^T(T)\,=\,0$ if and only $T$ is reflexive. Moreover the properness of $D^T$  is a direct consequence of
Lemma 4.7.1, and Lemma 4.7.2 of \cite{Weber1998TeichmullerTA}.

\subsection {Tangent vector at $t_0\in \mathcal{T}$}

\begin{figure}
    \centering
   \includegraphics[scale=0.8]{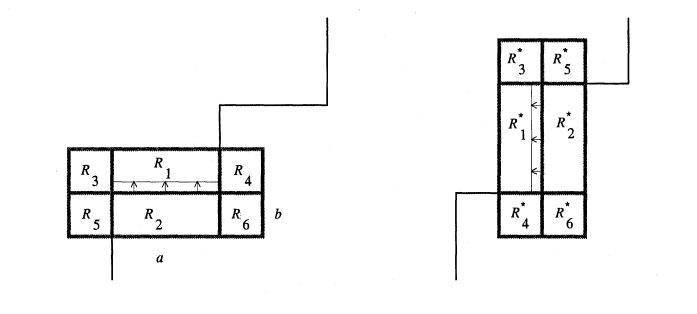}
    \caption{Picture from the Section 5.3.1 \cite{Weber1998TeichmullerTA}}
    \label{fig:FepsilonandFepsilon*}
\end{figure}
Let us fix an edge $E$ that is horizontal; the symmetric tweezer will give the corresponding edge $E^*$, which is vertical. 
For given $b, \delta$, take maps $f_\epsilon^E$ and $f_{\epsilon^*}^E$ as in \cite[(5.1)(a)]{Weber1998TeichmullerTA} and
\cite[(5.1)(b)]{Weber1998TeichmullerTA} respectively.
Geometrically, these maps push the horizontal and vertical edges so that after applying both maps, the tweezer structure
is preserved (see Figure~\ref{fig:FepsilonandFepsilon*}).  Under the reflection across the line $y\,=\,-x$, the region $R_i$ is
mapped to the region $R_i^*$ as in figure \ref{fig:FepsilonandFepsilon*}.

We call the map $f^E_\epsilon \circ f^{E}_{\epsilon^*}$ above the ``pushing out and pulling in''  map for the edge $E$.

Let $\nu_\epsilon \,:=\, \frac{(f_\epsilon)_{\overline{z}}}{(f_\epsilon)_{z}}$ represents the Beltrami differential of $f_\epsilon$, and 
define $\dot{\nu} \,= \,\left. \frac{d}{d\epsilon}\right|_{\epsilon=0}\nu_\epsilon$. Similarly, let $\dot{v^{*}}$ denote the infinitesimal 
Beltrami differential of $f^*_\epsilon$. Expressions for $\dot{\nu}$ and $\dot{\nu^{*}}$ are given in
\cite[(5.1)(a)]{Weber1998TeichmullerTA} and \cite[(5.1)(b)]{Weber1998TeichmullerTA} respectively.

We take $\dot{\mu} \,= \,\dot{\nu} + \dot{\nu}^*$. This is a Beltrami differential supported on a bounded domain in
$\mathbb{C}\,=\, \Omega_{Gdh}\cup t_0\cup\Omega_{G^{-1}dh}$. Thus, this pair of Beltrami differentials lifts to a pair 
\begin{equation}\label{tangentvector}
    \dot{\mu}\,=\,(\dot{\mu}_{\Omega_{Gdh}},\, \dot{\mu}_{\Omega_{G^{-1}dh}})
\end{equation}
on the pair $S_{\Omega{Gdh}}$ and $S_{\Omega_{G^{-1}dh}}$. 

The above defined $\dot{\mu}$ represents a tangent vector to $\mathcal{T}_p$ at $t_0$.  The above process will yield different tangent vectors for different ``pushing out and pulling in'' maps.

\subsection{Derivative of extremal length function, and corresponding quadratic differential}

Let $\Phi_k$ represent the quadratic differential associated with the homotopy class of curve $\alpha_k$, defined as
$$
\Phi_k\,:=\, \frac{1}{2} d\text{ Ext}([\alpha_k])|_{\Omega_{Gdh}}\ \  \text{ and, }
$$
\begin{equation}\label{eqn:derivativeofExtermal}
    \left(d\text{ Ext}([\alpha_k])|_{\Omega_{Gdh}}\right)[\widehat{\nu}] \,=\, 4\text{Re}\int_{\Omega_{Gdh}}
\Phi_k\widehat{\nu}.
\end{equation}
The horizontal foliation of $\Phi_k$ comprises curves connecting the same edges, as $C_k$, within $\Omega_{Gdh}$. Furthermore, it preserves the reflective symmetry of the element in $\mathcal{T}_p$. Consequently, these foliations must either run parallel to or be perpendicular to the fixed sets of reflections (which correspond to the tweezer).  Here, we are using the same notation for the foliation $\Phi_k$ both on the marked sphere $S_{\Omega_{Gdh}}$ and on $\Omega_{Gdh}$.

Consider the edge $E$ connecting vertices $v_1$ and $v_2$, but $v_i\,\neq\, P_0$. Let $\Phi_k$ denote the quadratic differential 
corresponding to the homotopy class of curve $C_k$ containing both $v_1$ and $v_2$.  Since genus $p\,\geq\, 2$, in this context, there is 
such an edge for which one of the following conditions holds for vertex $v_i$:
\begin{enumerate}
    \item The ortho-disk has an angle of $\frac{\pi}{2}$ at $v_i$.
    \item the vertex $v_i$ lies at infinity.
    \item The ortho-disk has an angle of $\frac{3\pi}{2}$ at $v_i$, and the foliation $\Phi_k$ is either parallel or orthogonal to the edges incident to $v_i$.
\end{enumerate} 

Using Proposition 5.3.2 from \cite{Weber1998TeichmullerTA} we conclude that the holomorphic quadratic differential $\Phi_k$ is 
admissible on the edge corresponding to points $v_1$ and $v_2$.

In the next subsection, we will select an edge $E$ such that the corresponding foliation is admissible on $E$. Additionally, we will 
consider the tangent vector $\dot{\mu}$ as described in \eqref{tangentvector}, which we obtain by applying the corresponding pushing out and pulling in function 
$f_\epsilon^E \circ f_{\epsilon^*}^E$ to the tweezer.  Without loss of generality, we can select the edge corresponding to $P_{-1}$ and 
$P_{-2}$.

\subsection{Reflexive tweezer}
Since the holomorphic quadratic differential $\Phi_k$  is admissible for the edge we selected (as demonstrated in the previous subsection), we can apply Lemma 5.3.1 from \cite{Weber1998TeichmullerTA}, leading to the following corollary for the tweezer case:
\begin{corollary}\label{sgnQuadraticDifferential}
The expression $sgn(\Phi_k \dot{\mu}_{\Omega_{Gdh}}(q))$ maintains a constant sign for $q \in E$, which is opposite to $sgn(\Phi_k \dot{\mu}_{\Omega_{G^{-1}dh}}(q))$. Consequently,

$$\text{sgn}\left( d\text{Ext}_{\Omega_{Gdh}}([\alpha_k])[\dot{\mu}_{\Omega_{Gdh}}] \right) = -\text{sgn}\left( d\text{Ext}_{\Omega_{G^{-1}dh}}([\alpha_k])[\dot{\mu}_{\Omega_{G^{-1}dh}}] \right).$$

\end{corollary}
In the above, we are employing the same notation (a slight abuse of notation) for the curve $\alpha_k$ in both $\Omega_{Gdh}$ and $\Omega_{G^{-1}dh}$, as well as the same symbol for the corresponding horizontal foliation.    

The above corollary \ref{sgnQuadraticDifferential} is the main result that enables us to use the findings of zigzags, as presented in Section 5.3 of \cite{weber1998minimal}. Since the proof is identical, we can establish the following fact: 

Suppose there is a reflexive tweezer of genus $p-1$, then there is a path (one dimensional real analytic manifold) $\mathcal{Y} \subset \mathcal{T}_p$ for which $\text{Ext}_{\Omega_{Gdh}}(\alpha_i) = \text{Ext}_{\Omega_{G^{-1}dh}}(\alpha_i)$ holds for $i = 2 \ldots p-1$.

If we restrict $D^T$ as in the Equation \eqref{defn: DT} to $\mathcal{Y}$, we have 
\[
    D^T|_{\mathcal{Y}} = \left(\exp\left(\frac{1}{\text{E}_{\Omega_{Gdh}}(1)}\right) - \exp\left(\frac{1}{\text{E}_{\Omega_{G^{-1}dh}}(1)}\right)\right)^2 + \left(\text{E}_{\Omega_{Gdh}}(1) - \text{E}_{\Omega_{G^{-1}dh}}(1)\right)^2
\]

We denote the above-restricted function by the same $D^T$. 
 Moreover let $t_0\in Y$,   and let $\dot{\mu}\in T_{t_0}\mathcal{Y}$ corresponding to the edge we selected. We get

\begin{eqnarray*}
dD^T_{t_0}[\dot{\mu}] &=& 2\left(\exp\left(\frac{1}{\text{Ext}_{\Omega_{Gdh}}([\alpha_1])}\right) - \exp\left(\frac{1}{\text{Ext}_{\Omega_{G^{-1}dh}}([\alpha_1])}\right)\right) \\
&& \left(-\exp\left(\frac{1}{\text{Ext}_{\Omega_{Gdh}}([\alpha_1])}\right)(\text{Ext}_{\Omega_{Gdh}}([\alpha_1]))^{-2}d\text{Ext}_{\Omega_{Gdh}}([\alpha_1])[\dot{\mu}_{\Omega_{Gdh}}] \right.\\
&& \left. +\exp\left(\frac{1}{\text{Ext}_{\Omega_{G^{-1}dh}}([\alpha_1])}\right)(\text{Ext}_{\Omega_{G^{-1}dh}}([\alpha_1]))^{-2}d\text{Ext}_{\Omega_{G^{-1}dh}}([\alpha_1])[\dot{\mu}_{\Omega_{G^{-1}dh}}]\right)\\
&& + 2\left(\text{Ext}_{\Omega_{Gdh}}([\alpha_1]) - \text{Ext}_{\Omega_{G^{-1}dh}}([\alpha_1])\right)\\
&& \left(d\text{Ext}_{\Omega_{Gdh}}([\alpha_1])[\dot{\mu}_{\Omega_{Gdh}}] - d\text{Ext}_{\Omega_{G^{-1}dh}}([\alpha_1])[\dot{\mu}_{\Omega_{G^{-1}dh}}]\right).
\end{eqnarray*}

If we begin with $t_0\in Y$ such that $D^T(t_0)\neq 0$ and choose $\dot{\mu}$ as described in \ref{tangentvector} for the edge we selected at the beginning of this subsection.  The equation above implies that $dD^T_{t_0}[\dot{\mu}]\neq 0$. However, since $D^T$ is a proper map, it must have a critical point in the smooth manifold $\mathcal{Y}$; therefore, at this critical point $t_0$, we must have $D^T(t_0)=0$.

Now, finally, we prove that for every genus $p$, there exists a reflexive tweezer using an induction argument as outlined in \cite{weber1998minimal} (page 1165). To apply the induction argument for the case when $p=1$, it is evident that both the Riemann surfaces resulting from the zigzag and the tweezer are square tori. Therefore, for $p=1$, the existence of reflexive zigzag and tweezer is immediate, and then we can proceed with the induction argument.

\section{Analyzing the singularities}

As described in Theorem \ref{prop:maxfacefromZigzag} and Theorem \ref{thm:surfaceFromTweezer}, we have 
two types of maxfaces with Enneper end. The aim in this section is to understand the singularities of these 
surfaces. Although the Weierstrass data of these maxfaces are not explicitly given, they depend on the 
existence of a reflexive zigzag and tweezer.
For maxfaces, the singular set is $\{p\,\in\, M\,\, \big\vert\,\, |g(p)|\,=\,1\}$. 
Therefore, we start by revisiting the Gauss map.

\subsection{The Gauss Map and the 1-form}
For the case of a symmetric zigzag (respectively tweezer) of genus $p$, we recall that $\Omega_{NE}$ (respectively $\Omega_{Gdh}$) is an ortho-disk, defined as $\mathbb{H}\cup\mathbb{R}$ equipped with the pullback of the flat metric on $\mathbb{C}$ by the respective Schwarz-Christoffel map $F$. This Schwarz-Christoffel map is defined as follows:

\begin{equation}\label{SCmapZigzag}
F(z) \,=\, \int_{i}^{z} \prod_{j=-p}^{p} (t-t_j)^{a_j} dt.
\end{equation}
Here, $a_j$ takes the values $\pm\frac{1}{2}$. For the case of the zigzag, it alternates in value, starting from $a_{-p} = -\frac{1}{2}$. Moreover, for the tweezer, it does not alternate; instead, there are three consecutive points where the value of $a_j$ is the same.  

In the remaining section, we will be using $\Omega_{Gdh}$ for the ortho-disk obtained from either the tweezer or the zigzag.  Now, let $\{\mathcal{R}_{\Omega_{Gdh}}\setminus\{P_{\infty}\},\, g,\, dh\}$ be a 
maxface either as described in Theorem \ref{prop:maxfacefromZigzag} or in Theorem \ref{thm:surfaceFromTweezer}. 


As in the previous section, denote by $S_{\Omega_{G dh}}$ the double of $\Omega_{G dh}$ with marked points $t_j$ and puncture at $\infty$.  Consider the  hyperelliptic double covering map $\pi\,:\, \mathcal{R}_{NE}\,\longrightarrow\, S_{\Omega_{G dh}}$ whose branched points  are $\{t_j\}$ and $\infty$. On $\pi^{-1}(\Omega_{\Omega_{G dh}})$, the Gauss map $g$ and a $1$-form $dh$ are given by the 
following expressions:

\begin{eqnarray*}
    g(z,w) &\,\,=\,\,&\frac{e^{\pm\frac{ i \pi }{4}}}{c}\prod_{j=-p}^{p}(z-t_j)^{a_j}\\
    {dh}&\,\,=\,\,& cdz.
\end{eqnarray*}

We do not have an explicit knowledge of $t_j$ and $c$; their existence is linked with
the existence of the reflexive zigzag. The Gauss map $g$ and $dh$ can be extended to $\mathcal{R}_{\Omega_{G dh}}$
by using the symmetries: rotation and reflections. Therefore, if we determine the singular set in this component, we will gain insights into the full singular set on $\mathcal{R}_{\Omega_{G dh}}$  of the maxface by applying reflections and rotations on it. 

The singular set in the components $\Omega_{{G dh}}$ and its mirror image with respect to the real axis
are given as follows:
\begin{equation}\label{Singular set}
Sing_{\geq} \,:=\, \lbrace (z,\,w)\,\,\big\vert\,\, \text{Im}(z) \,\geq\, 0,\;\;\text{and}\;\;
\prod_{j=-p}^{p}\left\lvert z-t_j\right\rvert^{a_j} \,=\, \left\lvert c\right\rvert\rbrace
\end{equation}
\begin{equation}
Sing_{\leq} \,:=\, \lbrace (z,\,w)\,\,\big\vert\,\, \text{Im}(z)\,\leq\, 0,\;\;\text{and}\;\;
\prod_{j=-p}^{p} \left\vert z-t_j\right\rvert^{a_j} \,=\, \left\lvert c\right\rvert\rbrace.
\end{equation}
Denote by $S$ the union $S_{\geq}\cup S_{\leq}.$  The singular set of maxfaces in $\mathcal{R}_{\Omega_{G dh}}$ is $\pi^{-1}(S)$. To
understand $S$ and $\pi^{-1}(S)$, we introduce the function:
\begin{equation}\label{eqn:G}
G\,:\, \mathbb{C}\cup\lbrace \infty\rbrace \,\longrightarrow\, \mathbb{C}\cup\lbrace \infty\rbrace,\,\ \
z\, \longmapsto\, \frac{\prod_{j=-p}^{p}( z-t_j)^{2a_j}}{c^2}.
\end{equation}

It is evident that if we set $S_1 \,= \,G^{-1}\{|z|=1\}$, then $S_1 \,=\, S
\,=\, \text{Sing}_{\geq} \cup \text{Sing}_{\leq}$. The poles of $G$ are of order 1, so poles are not
ramification points. Therefore, $G$ is a branched covering with ramification points in 
the zeros of $G^\prime$; the locus of the zeros of $G^\prime$ is denoted by $Z(G^\prime)$.

In general, the structure of the singular set depends on the values of 
$t_j$ and $c$. After experimenting with various values of $c$ and $t_j$ in Mathematica, it appears 
that the singular set is a collection of topological circles, sometimes disjoint and at other times 
intersecting. In Figure \ref{fig:singularysetIDEA}, we can see various possibilities.

\begin{figure}[h]
    \centering
    \def\svgwidth{0.45\columnwidth}
\begingroup%
  \makeatletter%
  \providecommand\color[2][]{%
    \errmessage{(Inkscape) Color is used for the text in Inkscape, but the package 'color.sty' is not loaded}%
    \renewcommand\color[2][]{}%
  }%
  \providecommand\transparent[1]{%
    \errmessage{(Inkscape) Transparency is used (non-zero) for the text in Inkscape, but the package 'transparent.sty' is not loaded}%
    \renewcommand\transparent[1]{}%
  }%
  \providecommand\rotatebox[2]{#2}%
  \newcommand*\fsize{\dimexpr\f@size pt\relax}%
  \newcommand*\lineheight[1]{\fontsize{\fsize}{#1\fsize}\selectfont}%
  \ifx\svgwidth\undefined%
    \setlength{\unitlength}{360bp}%
    \ifx\svgscale\undefined%
      \relax%
    \else%
      \setlength{\unitlength}{\unitlength * \real{\svgscale}}%
    \fi%
  \else%
    \setlength{\unitlength}{\svgwidth}%
  \fi%
  \global\let\svgwidth\undefined%
  \global\let\svgscale\undefined%
  \makeatother%
  \begin{picture}(1,1.00277778)%
    \lineheight{1}%
    \setlength\tabcolsep{0pt}%
    \put(0,0){\includegraphics[width=\unitlength,page=1]{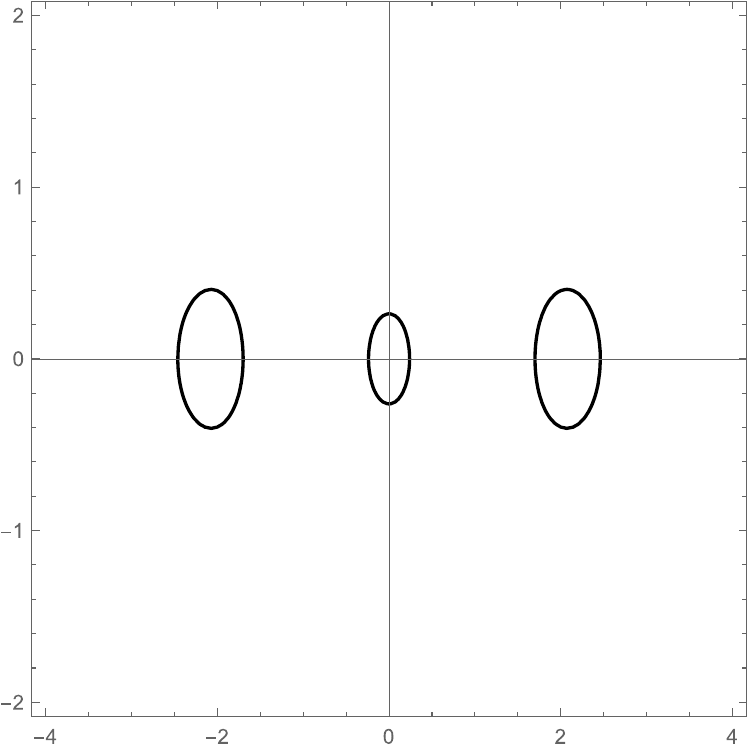}}%
    \put(0.18477287,0.15975155){\color[rgb]{0,0,0}\transparent{0.99581599}\makebox(0,0)[lt]{\lineheight{1.25}\smash{\begin{tabular}[t]{l}$|z^2 - 1|=|z (z^2 - 4)|$\end{tabular}}}}%
  \end{picture}%
\endgroup%

, %
    \def\svgwidth{0.45\columnwidth}
\begingroup%
  \makeatletter%
  \providecommand\color[2][]{%
    \errmessage{(Inkscape) Color is used for the text in Inkscape, but the package 'color.sty' is not loaded}%
    \renewcommand\color[2][]{}%
  }%
  \providecommand\transparent[1]{%
    \errmessage{(Inkscape) Transparency is used (non-zero) for the text in Inkscape, but the package 'transparent.sty' is not loaded}%
    \renewcommand\transparent[1]{}%
  }%
  \providecommand\rotatebox[2]{#2}%
  \newcommand*\fsize{\dimexpr\f@size pt\relax}%
  \newcommand*\lineheight[1]{\fontsize{\fsize}{#1\fsize}\selectfont}%
  \ifx\svgwidth\undefined%
    \setlength{\unitlength}{360bp}%
    \ifx\svgscale\undefined%
      \relax%
    \else%
      \setlength{\unitlength}{\unitlength * \real{\svgscale}}%
    \fi%
  \else%
    \setlength{\unitlength}{\svgwidth}%
  \fi%
  \global\let\svgwidth\undefined%
  \global\let\svgscale\undefined%
  \makeatother%
  \begin{picture}(1,1.00277778)%
    \lineheight{1}%
    \setlength\tabcolsep{0pt}%
    \put(0,0){\includegraphics[width=\unitlength,page=1]{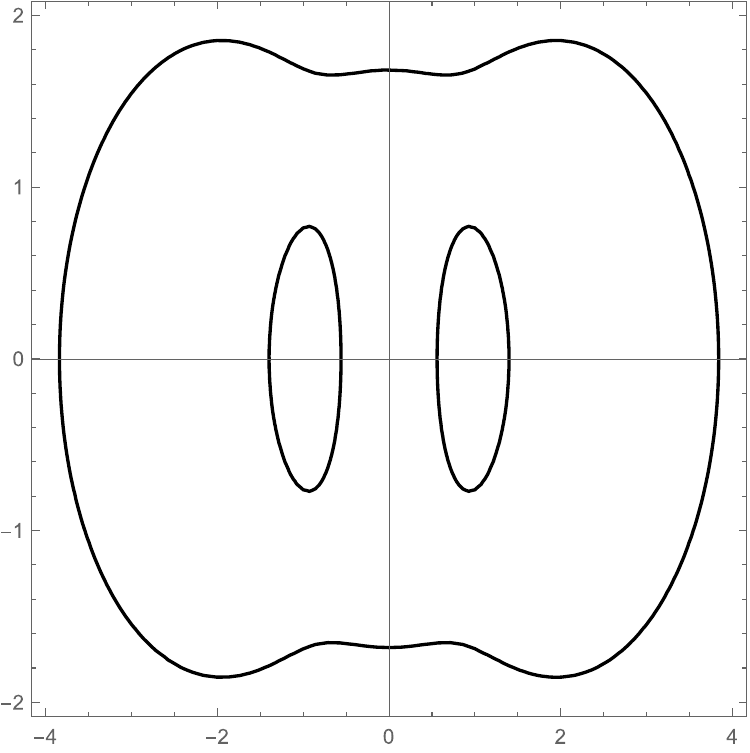}}%
    \put(0.11337999,0.06533856){\color[rgb]{0,0,0}\transparent{0.99581599}\makebox(0,0)[lt]{\lineheight{1.25}\smash{\begin{tabular}[t]{l}$3|z^2 - 1|=|z (z^2 - 4)|$\end{tabular}}}}%
  \end{picture}%
\endgroup%

    \caption{The singularity set may be topological circles or concentric circles, intersecting or disjoint.}
    \label{fig:singularysetIDEA}
\end{figure}

In the following, we write the general things that we can say: 
\begin{prop}\label{Prop:SingularsetONRNE}
   For genus $p$ maxfaces from the zigzag (as described in Theorem \ref{prop:maxfacefromZigzag}), 
if $G^\prime(z)\,\neq\, 0$ on the singular set, then the singular set of the maxface on the defining Riemann surface consists of  $(p+1)$ disjoint loops.
\end{prop}

\begin{proof}
As $G^\prime \,\neq\, 0$ on the singular set, the connected component of the preimage of $\{|z|\,=\,1\}$ is a 
smooth one-dimensional object and these connected components are non-intersecting.

Define $G_1$ on $(t_j,t_{j+1})$ as $G_1(x) \,:=\, |G(x)| - 1$.  For the zigzag case, it is important to note that $t_{-j} = -t_j$ and the sign of $a_j$ alternates.  Therefore, if  $t_j$ is a pole of $G$, then $t_{j+1}$ is 
zero, or vice versa. Therefore, between $t_j$ and $t_{j+1}$, there exists a real number such that $G_1(z) \,=\, 
0$. Consequently, there are at least $2(p+1)$ real points $r_i$ that are singular points. Moreover, 
$\lbrace x\,\in\,\R\vert G_1(x)\,=\,0\rbrace$ is contained in the solution set of a $2p+2$ degree real polynomial; 
thus, there are exactly $2p+2$ real points, which are singular points\color{black}. We take preimages  of 
the loop $\{|z|\,=\,1\}$ starting from $r_i$, denoting these preimages as $\gamma_i$. The traces of these 
$\gamma_i$ are part of the singular set.

Since the degree of $g$ is $p+1$, we cannot have more than $p+1$ disjoint components of the singular set. 
Therefore, each $\gamma_i$ intersects at least one $\gamma_j$, forming part of the same connected 
component.  Moreover, each connected component is one-dimensional and symmetric about the $X$-axis; it must 
be a closed loop.

All such loops are $p+1$ in number within the double of $\Omega_{{G dh}}$. Moreover, these loops either 
encircle only one zero of $G$ or only one the pole of $G$. Thus, when we lift these singular locus in the hyperelliptic cover 
$\mathcal{R}_{\Omega_{G dh}}$, we obtain exactly $p+1$ disjoint singular locus.
\end{proof}

We refer to the loops that constitute connected components of singularities as singular locus. The crucial point 
in the proof of the above proposition is that we have used the poles and zeros of $G$ in the case of zigzag 
alternates. However, in the case of a tweezer, this is not applicable. For the case of the reflexive tweezer, where there are exactly three points where the sign of $a_j$ does not alternate, we can apply the same argument as above to conclude that there will be at least $p-1$ disjoint singular locus when $p\geq 2$. When $p=1$, both the maxface from the tweezer and zigzags are the same.

\begin{figure}[h]
    \centering
    \includegraphics[scale=0.56]{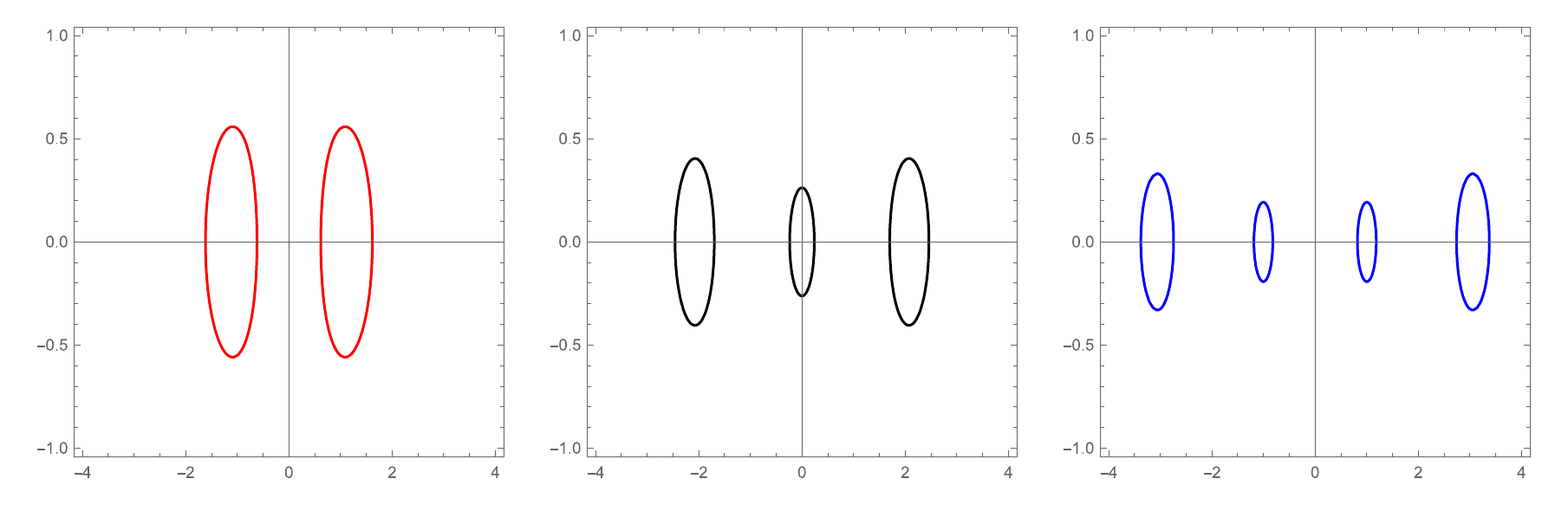}
    \caption{Singularity set in $S_{\Omega_{G dh}}$ for the case of zigzag, when $c=1$ and points $t_j=j$, with examples for genus 1, genus 2, and genus 3.
 }
    \label{fig:disjointsingularcurve}
\end{figure}

\subsection{The nature of singularities}

The nature of the singular points is characterized by three functions.
They are $\mathrm{A}(z)\,=\,\frac{g^{\prime}}{gf}$,\, $B(z)\,=\,\frac{g}{g^{\prime}}A^{\prime}$, and
$E(z)\,=\,\frac{g}{g^{\prime}}B^{\prime}$.   We recall the criterion \cite{UMEHARA2006,SaiPradip} to check the nature of singularity in the Table \ref{tab:Singuarties: in terms of }: 
\begin{table}[h]
\begin{tabular}{|c c c c c|}
\hline
  Re$(A)\neq 0$ & Im$(A)\neq 0$ & & &\hspace{-1cm}$\Leftrightarrow p$ is a cuspidal-edge\\\hline
  Re$(A)\neq 0$ &Im$(A)= 0$ & Re $(B)\neq 0$& &\hspace{-1cm}$\Leftrightarrow p$ is a swallowtail\\\hline
Re$(A)\neq 0$ & Im$(A)= 0$ & Re $(B)= 0$ &Im$(E)\neq 0$&$\Leftrightarrow p$ is a cuspidal butterfly\\\hline
Re$(A)\neq 0$ & Im$(A)= 0$ & Re $(B)= 0$ &Im$(E)= 0$&$\Leftrightarrow p$ is a special singular point of type 1\\\hline
Re$(A)=0$ & Im$(A)\neq 0$ & Im$(B)= 0$ &Re$(E)\neq 0$&$\Leftrightarrow p$ is a cuspidal $S_1^-$\\\hline
Re$(A)=0$ & Im$(A)\neq 0$ & Im$(B)= 0$ &Re$(E)=0$&$\Leftrightarrow p$ is a special singular point of type 2\\\hline
Re$(A)= 0$ & Im$(A)\neq 0$ & Im $(B)\neq 0$& &$\Leftrightarrow p$ is a cuspidal crosscap\\\hline
\end{tabular}
\caption{Singularities: in terms of $A$,$B$, and $E$}
\label{tab:Singuarties: in terms of }
\end{table}

In our situation, we have $A(z)\,=\, \frac{g^\prime(z)}{g(z)f(z)}\,=\, \frac{g^\prime}{g c}.$   On the singular set, when $g(z)\neq 0$,  we have $A(z)\,=\,
\frac{G^\prime(z)}{2g(z)}.\frac{1}{g(z).c}\,=\, \frac{G^\prime(z)}{2c.G(z)}.$ 
We will discuss the case where all $z_0\,\in\, S$ (singular set) satisfy  $G^\prime(z_0)\,\neq\, 0$;
this implies that $A\,\neq\, 0$ on the singular set. 

Therefore, for each point in the singular set, either $\mathrm{Re}(A)\,\neq \,0$ or $\mathrm{Im}(A)\,
\neq\, 0$.  

\subsubsection{\textbf{$\mathrm{Re}(A)$} and $\mathrm{Im}(A)$}

Let $z$ be a solution such that \( \mathrm{Re}(A(z)) \,=\, 0 \). Since \( A(z) \,=\,
\frac{G^\prime(z)}{2c\cdot G(z)} \,=\, \frac{1}{2c}\left( \sum_{k=1}^p\frac{ 4\; a_k\; z}{(z^2-t_k^2)}+
\frac{2a_0}{z}\right) \)\color{black}, we have \( \mathrm{Re}(A)(z) \,=\, 0 \) if and only if the
following equation is satisfied:
\begin{equation}\label{ReAlpha}
H(z) \,:=\, \mathrm{Re} \left( \frac{1}{2c} \left(\sum_{k=1}^p\frac{ 4\; a_k\; z}{(z^2-t_k^2)}+
\frac{2a_0}{z}\right)\right) \,=\, 0.
\end{equation}

Take \( H_1(z) \,=\, -\frac{1}{2c} \left( \sum_{k=1}^p\frac{ 4\; a_k\; z}{(z^2-t_k^2)}+
\frac{2a_0}{z} \right) \); it is a function from \( \mathbb{C}\cup\{\infty\}\) to itself.
Let \( C_k \) be the loop in the singular locus that encircles \( t_k \). Moreover, assume that \( t_k \) is
a zero of \( G \) (the same analysis holds when $t_k$ is the pole). We wish to show that the cardinality of the
intersection satisfies the following:
\[ |C_k \cap H_1^{-1}\left(\{iy\,\big\vert\,\, y\in\mathbb{R}\}\cup\{\infty\}\right)| \,\geq\, 2. \]
It is clear that \( H_1 \) takes \( t_k \) to \( \infty \) and \( \infty \) to \( 0 \). This means that
the preimage should contain \( \infty \), which implies that a part of the connected component of
$H_1^{-1}\left(\{iy\,\big\vert\,\, y\in\mathbb{R}\}\cup\{\infty\}\right)$ should lie outside the loop. On
the other hand, \( t_k \), which is inside the loop, also lies in the preimage. This means that a connected
component of the preimage will intersect \( C_k \) at least once. By symmetry, it will intersect
$C_k$ atleast twice in the double of the \( \Omega_{\Omega_{G dh}} \). Therefore, we have
\( |C_k \cap H_1^{-1}\left(\{iy\,:\,\, y\in\mathbb{R}\}\cup\{\infty\}\right)| \,\geq\, 2 \).

Similarly, we can prove that there are at least two points on $C_k$ where $Im(A(z))$ is not zero.

We recall that the data for the zigzag is given by \((c,t_1,\,\cdots,\,t_p)\) with \(t_{-j}\,=\,-t_j\). Similarly, 
if we assume that the data for the tweezers also satisfies the condition \(t_{-j}\,=\,-t_j\), then the above 
analysis holds for the tweezer case.

Combining all these, we have the following proposition.

\begin{prop}\label{propSigng3}
For the maxface from the zigzag, or tweezer, let \(G\) be such that \(G^\prime \,\neq\, 0\) on any singular loop \(C_k\)
in the singular locus. Then, on each \(C_k\), we have:
\begin{enumerate}
    \item There are at least two points on \(C_k\) such that \(\mathrm{Re}(A)(z)\, =\, 0\)
and \(\mathrm{Im}(A)(z) \,\neq\, 0\).

\item There are at least two points on \(C_k\) such that \(\mathrm{Re}(A)(z) \,\neq\,
 0\) and \(\mathrm{Im}(A)(z) \,= \,0\).
\end{enumerate}
Furthermore, if the tweezer has data \((c,\,t_1,\,\cdots,\,t_p)\) with \(t_{-j}\,=\,-t_j\), then both
(1) and (2) above hold true for the maxface from the tweezers as well.
\end{prop}

\begin{figure}[h]
    \centering
    \includegraphics[scale=0.8]{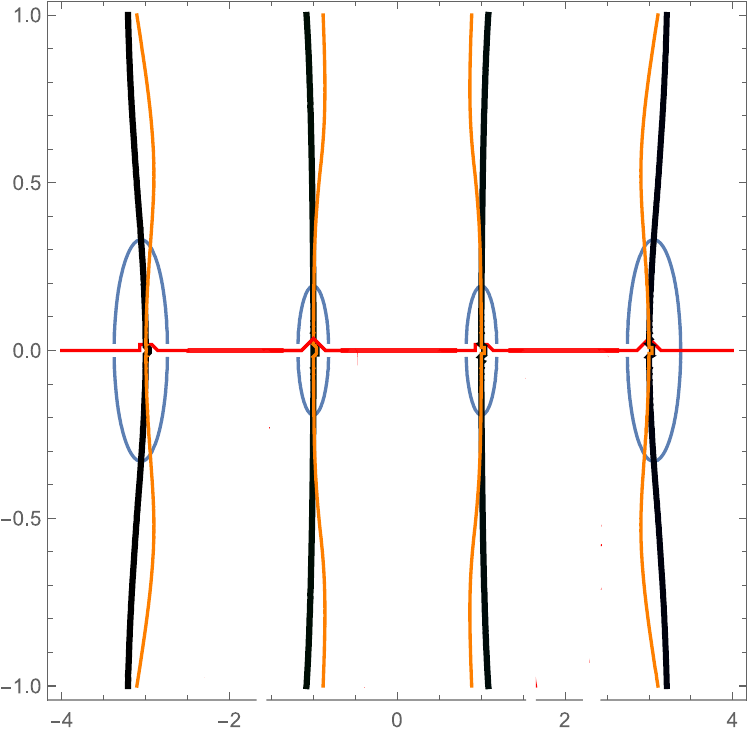}
    \caption{For genus three, when $t_1=1$, $t_2=2$, $t_3=3$, and $c=1$. Blue topological circles represent a singular locus.  The red curves show the points where $\mathrm{Im}(A)=0$, the black curves represent where $\mathrm{Re}(A)=0$, and the orange curves are where $\mathrm{Re}(B)=0$.}
    \label{fig:figlabel}
\end{figure}

\subsubsection{Re($B$) and Im($B$)}
We calculate $B(z)$ on $\Omega_{\Omega_{G dh}}$ and on the double of it. We get 
\[ B(z)\, =\, \frac{(M^\prime(z) \cdot G(z) - M(z) \cdot G^\prime(z))}{c^3 \cdot G(z) \cdot G^\prime(z)}. \]
Here $M(z)$ is given by:
\[ M(z) \,=\, \sum_{k=-p}^p 2a_k\prod_{j=-p, j\neq k}^p (z-t_j)^{2a_j}(z-t_k)^{2a_k-1}. \]

This equation for $B(z)$ is in terms of polynomial $t_k$, so again, using the same method as for 
$\mathrm{Re}(A)$ and $\mathrm{Im}(A)$, we see that the inverse image of the real axis and the imaginary 
axis will intersect $C_k$ at least twice.

\subsection{Final result for the nature of singularities}

We consider a maxface, which is a front as a map from $\mathcal{R}_{\Omega_{G dh}}$ to $\mathbb{R}^3$.  We refer to 
\cite{UMEHARA2006} for the definition of ``front" and ``frontal".  Umehara and Yamada in \cite{Umehara2011} 
discussed the various singularities that can appear on a maxface when it is a front. We know 
(\cite{UMEHARA2006}) that if a maxface is a front, then we must have $A(z) \,\neq\, 0$ on the singular set. 
This implies that $ G^\prime(z) \,\neq \,0$ on the singular set.

 In Table \ref{tab:Singuarties: in terms of }, 
we list some of the singularities considering the maxface as a front or frontal.  Combining the above discussion with Propositions \ref{Prop:SingularsetONRNE} and 
\ref{propSigng3}, we present the following result regarding the nature of singularities in the surfaces as in Theorem \ref{prop:maxfacefromZigzag} and Theorem \ref{thm:surfaceFromTweezer}.

\begin{theorem}\label{thm:Sing}
Let $X$ be a maxface with genus $p$ defined via a zigzag. Moreover, if this maxface is a front, then it
possesses $(p+1)$ connected components of the singular set, each of which is topologically a
circle. Each component will have:
\begin{enumerate}
    \item At least two points where we have either swallowtails, cuspidal butterflies, or the special singular points of type 1.

    \item At least two points where we have either cuspidal cross-caps, cuspidal $S_1^-$, or special singular points of type 2.

    \item Other singularities that are not any of the above are cuspidal edges.

\end{enumerate}
Furthermore, if the tweezer is of genus $1$, then the maxface from the tweezer and zigzag are the same. For genus $p\geq 2$, if the tweezer has data \((c, t_1, \ldots, t_p)\) with \(t_{-j} = -t_j\), then it possesses at least $p-1$ connected components of the singular set. Moreover, the nature of the singularities in each connected component of singularities is the same as in the case of the zigzag.

\end{theorem}

\begin{figure}
    \centering
    \includegraphics[scale=0.3]{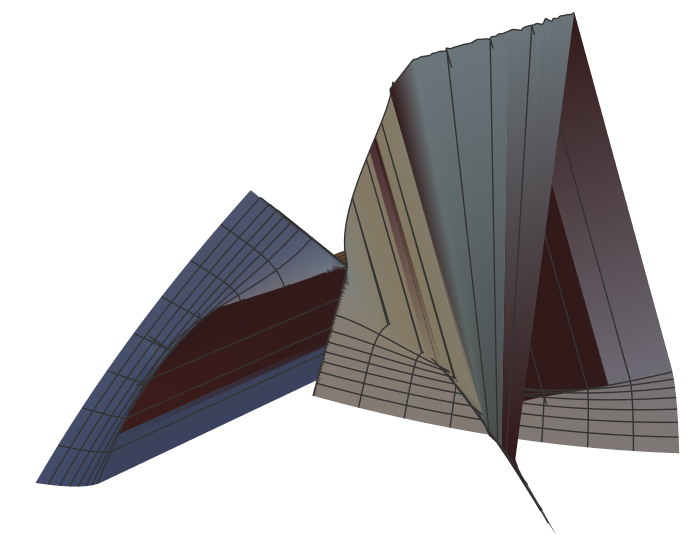}
    \def\svgwidth{0.45\columnwidth}

    \caption{Half of Lorentzian Chen-Gackstatter surface; view from two sides.}
    \label{fig:enter-label}
\end{figure}

\bibliography{ref.bib}

\end{document}